\pdfoutput=1
\documentclass[journal]{IEEEtran}

\usepackage{array}
\usepackage{amsmath}
\usepackage{amsfonts}
\usepackage{amssymb}
\usepackage{mathrsfs}
\usepackage{latexsym}
\usepackage[dvips]{epsfig}
\usepackage[font=footnotesize]{subfig}
\usepackage{hyperref}
\usepackage{graphics}

\usepackage{enumerate}

\usepackage{amsthm}
\usepackage{arydshln}
\usepackage{cite}

\usepackage{listings}
\usepackage{color} 
\usepackage{verbdef}
\definecolor{mygreen}{RGB}{28,172,0} 
\definecolor{mylilas}{RGB}{170,55,241}
\usepackage{csquotes}
\usepackage{bibentry}
\usepackage{mathtools}
\usepackage{url}  

\usepackage{amsmath,amssymb,amsfonts,mathrsfs,bm} 
\usepackage{caption}
\usepackage{cleveref}
\usepackage{array}
\usepackage{scalerel}
\usepackage{algorithm}
\usepackage{booktabs}
\usepackage{multirow}
\usepackage{rotating}
\usepackage{algpseudocode}
\usepackage{cite}
\usepackage{epsfig}
\usepackage{tabularx}
\usepackage{epstopdf}
\usepackage{circuitikz}
\usepackage{textcomp}
\usepackage{pgfplots}
\usepackage{lscape}
\pgfplotsset{compat=1.14}
\usepackage{array,booktabs}
\usepackage{siunitx}

\usepackage{lipsum}
\usepackage{mathtools}
\usetikzlibrary{arrows,shapes,positioning}

\usepackage{xspace}

\ifx\theorem\undefined
\newtheorem{theorem}{Theorem}
\newenvironment{theorem*}{\par\noindent{\bf Theorem\ }}{\hfill\\[2mm]}

\newtheorem{lemma}{Lemma}

\newenvironment{corollary*}{\par\noindent{\bf Corollary\ }}{\hfill\\[2mm]}
\newtheorem{definition}{Definition}

\fi

\newcommand{\Abf}{\boldsymbol{A}}

\newcommand{\Acal}{\mathcal{A}}
\newcommand{\bbf}{\boldsymbol{b}}

\newcommand{\cbf}{\boldsymbol{c}}
\newcommand{\Ccal}{\mathcal{C}}

\newcommand{\dbf}{\boldsymbol{d}}
\newcommand{\Dbf}{\boldsymbol{D}}

\newcommand{\Ebf}{\boldsymbol{E}}

\newcommand{\Ibf}{\boldsymbol{I}}

\newcommand{\Kcal}{\mathcal{K}}

\newcommand{\Lcal}{\mathcal{L}}

\newcommand{\Nbb}{\mathbb{N}}

\newcommand{\Ob}{\boldsymbol{O}}
\newcommand{\obf}{\boldsymbol{o}}

\newcommand{\pbf}{\boldsymbol{p}}

\newcommand{\Rbb}{\mathbb{R}}

\newcommand{\sbf}{\boldsymbol{s}}

\newcommand{\ubf}{\boldsymbol{u}}

\newcommand{\rbf}{\boldsymbol{r}}
\newcommand{\vbf}{\boldsymbol{v}}

\newcommand{\wbf}{\boldsymbol{w}}

\newcommand{\xbf}{\boldsymbol{x}}

\newcommand{\ybf}{\boldsymbol{y}}

\newcommand{\zbf}{\boldsymbol{z}}

\newcommand{\zerobf}{\boldsymbol{0}}

\definecolor{Fcolor}{rgb}{0, 0.5, 0.25}
\newif\ifcomment

\crefrangelabelformat{equation}{(#3#1#4)\,--\,(#5#2#6)}
\crefmultiformat{equation}{(#2#1#3)}{\,--\,(#2#1#3)}{#2#1#3}{#2#1#3}
\crefname{equation}{}{}

\crefmultiformat{table}{(#2#1#3)}{\,--\,(#2#1#3)}{#2#1#3}{#2#1#3}
\crefname{table}{}{}%
\crefname{figure}{Figure}{Figures}
\crefname{algorithm}{Algorithm}{}
\crefname{table}{Table}{Tables}
\crefname{lemma}{Lemma}{Lemmas}
\crefname{theorem}{Theorem}{Theorems}
\crefname{section}{Section}{Sections}
\crefname{definition}{Definition}{Definitions}
\commentfalse

\makeatletter
\DeclareRobustCommand{\cev}[1]{%
	\mathpalette\do@cev{#1}%
}
\newcommand{\do@cev}[2]{%
	\fix@cev{#1}{+}%
	\reflectbox{$\m@th#1\vec{\reflectbox{$\fix@cev{#1}{-}\m@th#1#2\fix@cev{#1}{+}$}}$}%
	\fix@cev{#1}{-}%
}
\newcommand{\fix@cev}[2]{%
	\ifx#1\displaystyle
	\mkern#2 1mu
	\else
	\ifx#1\textstyle
	\mkern#2 3mu
	\else
	\ifx#1\scriptstyle
	\mkern#2 2mu
	\else
	\mkern#2 2mu
	\fi
	\fi
	\fi
}

\graphicspath{{images/}}
\begin{document}
%
\title{Rapid Convergence of First-Order Numerical\\ Algorithms via Adaptive Conditioning}
%

\author{Muhammad Adil, Sasan Tavakkol, and Ramtin Madani
\thanks{Muhammad Adil and Ramtin Madani are with the University of Texas at Arlington. Sasan Tavakkol is with Google Research. This work is funded, in part, by the Office of Naval Research under award N00014-18-1-2186, and approved for public release under DCN\# 43-7474-20.}}

\maketitle

\IEEEpeerreviewmaketitle

\begin{abstract}
This paper is an attempt to remedy the problem of slow convergence for first-order numerical algorithms by proposing an adaptive conditioning heuristic. First, we propose a parallelizable numerical algorithm that is capable of solving large-scale conic optimization problems on distributed platforms such as {graphics processing unit}  with orders-of-magnitude time improvement. Proof of global convergence is provided for the proposed algorithm. We argue that on the contrary to common belief, the condition number of the data matrix is not a reliable predictor of convergence speed. In light of this observation, an adaptive conditioning heuristic is proposed which enables higher accuracy compared to other first-order numerical algorithms. Numerical experiments on a wide range of large-scale linear programming and second-order cone programming problems demonstrate the scalability and computational advantages of the proposed algorithm compared to commercial and open-source state-of-the-art solvers.
\end{abstract}

\section{Introduction}

Conic optimization is of practical interest in a wide variety of areas such as operation research, machine learning, signal processing and optimal control. 
For this purpose, interior point-based algorithms perform very well and have become the standard method of solving conic optimization problems \cite {Meh92,DCB13, TTT03}. Various commercial and open-source solvers such as MOSEK \cite{mosek}, GUROBI \cite{gurobi}, and SeDuMi \cite{sedumi} are based on interior point methods as their default algorithm. Although interior-point methods are robust and theoretically sound, they do not scale well for very large conic optimization programs. 
Computational cost, memory issues, and incompatibility with distributed platforms are among the major impediment for interior point methods in solving large-scale and practical conic optimization problems. 
 
In recent years, operator splitting methods such as Douglas-Rachford Splitting (DRS) \cite{EB92, GB17, FZB19,SMP19 } and Alternating Direction Method of Multipliers (ADMM) \cite{BPC+11,FB2018, OCPB16,  ZFP+20, SBGB2018, MKL18, MKL15} have received particular attention because of their potential for parallelization and ability to scale. 
First order methods are popular because {the} iterative steps are computationally cheap and easy to implement and thus ideal for large scale problems where high accuracy solutions are typically not required.  Operator splitting techniques, {on the other hand} can lead to parallel and distributed implementation and provide moderate accuracy solutions to conic programs in a relatively lower computational time. 

Motivated by the cheap per iteration cost and ability to handle large scale problems, several first order operator splitting algorithms have been proposed recently. Authors in \cite{OCPB16}, introduce a solver (SCS), a homogeneous self-dual embedding method based on ADMM to solve large convex cone programs and provide primal or dual infeasibility certificates when relevant.  A MATLAB solver CDCS \cite{ZFP+20} extended the homogeneous self-dual embedding concept \cite{OCPB16} and exploits the sparsity structure using chordal decomposition for solving large scale semidefinite programming problems.  The ADMM algorithm introduced in \cite{BPC+11} is improved by selecting the proximal parameter and pre-conditioning to introduce an open-source software package called POGS (Proximal Graph Solver) \cite{FB2018} and multiple practical problems are tested to evaluate the performance. Another application of operator splitting methods is provided in an open-source solver OSQP (operator splitting solver for quadratic programs) \cite{SBGB2018}, where operator splitting technique is applied to solve quadratic programs. Open-source Julia implemented conic operator splitting method (COSMO) \cite{CGG20}, solves the quadratic objective function under conic constraints. In \cite{FZB19}, a Python package Anderson accelerated Douglas-Rachford splitting (A2DR) is introduced to solve large-scale non-smooth convex optimization problems. Although these solvers scale very well as the dimension of the problem increases in different practical areas but suffers from slow convergence and do not perform well when the given problem is  ill-conditioned \cite{SMP19, TP19, XTH17}.

First order methods are considered very sensitive to condition number of problem data and parameter selection, and consequently have limitations in achieving higher accuracy within a reasonable number of iterations \cite{SBGB2018, GTS2015, NLR+15, GB17}.  Although first order operator splitting methods have been studied extensively in recent years for solving large scale conic programs for different applications but until the recent past, very few efforts are made to study the convergence rate \cite{HL2017, OPY+2020}.  
As an attempt to solve the convergence rate issues, recently serious efforts have been made to make first order algorithms more robust and  practical  for real-world applications \cite{TP19, WS17, NLR+15,EY15, GHY17, DZ16, BG16, BG18}.  A line search method is proposed in \cite{GB2016} to accelerate convergence. In \cite{Deng2016}, a global linear convergence proof is given under strict convexity and Lipschitz gradient condition on one function.  A global linear convergence approach and metric selection approach shown in \cite{GB17} under strong convexity and smoothness conditions. Researchers have proposed several acceleration techniques to expedite the convergence speed of ADMM. Adaptive penalty scheme is introduced in \cite{XTH17, CSV16} to automatically tune the penalty parameter. In \cite{YBJ18, JYO19}, Anderson acceleration (AA) is applied to improve the convergence of local-global solver and ADMM with application to geometry optimization and physics simulation problems. The authors in \cite{ZOB18}, applied the type-I variant of Anderson acceleration \cite{FS09} to splitting conic solver (SCS) \cite{OCPB16} to solve conic optimization problems and improved the terminal convergence. A new framework known as SuperSCS is introduced in \cite{SMP19} by combining SCS solver with original type-II AA to solve large cone problems and it is shown that the new approach performs better than the original SCS solver.  Type-II Anderson acceleration Douglas-Rachford splitting (A2DR) algorithm is proposed in \cite{FZB19}, to show the rapid convergence or provide infeasibility/unboundedness certificates. However, most of these techniques works reasonably well under limited scenarios, particular conditions, and for a very specific problem structures and yield no tangible benefits for any general class of problems. Improvements from these techniques are very limited and has very mild effect on the convergence due to the nature of accelerated algorithms. Moreover, these techniques {fail}  to achieve a higher accuracy.


 {Operator splitting methods heavily} rely on the input problem data matrices, pre-conditioning, solution polishing, and step size parameter selection \cite{ ZFP+20, NLR+15, FB2018}. Parameter selection for global convergence is still a challenge to be addressed \cite{SBGB2018, GB17}.  {Despite the scalability and computational advantages}, these methods suffer from slow terminal convergence, and {are} highly sensitive to problem condition number, hence, cannot be applied to many practical problems \cite{EY15, BG18, FZB19}. 
There is a dire need to {develop} a general purpose, and reliable first order algorithm that encapsulates the benefits of simple inexpensive iterations and scaling properties of first order algorithm, as well as providing the highly reliable and accurate solutions similar to that of interior point methods.


In this work, we first show that the {condition number of data matrices has} no significant effect on convergence of general first order methods and this is the major impediment for achieving highly accurate results with operator splitting methods. Furthermore, we propose a new operating splitting method where each iteration requires simple arithmetic operations, leads to parallel and distributive implementation, {scales} gracefully for very large cone programs  and {provides} a very accurate solutions which is beyond the reach of other first order solvers. Moreover, in conjunction with massively parallelizable and cheap iterative algorithm  we propose a heuristic policy to scale the data matrices in such a way that the combined algorithm ensures the global convergence and achieves a high accuracy within a tens of iterations.  In short, the proposed algorithm enjoys the benefits of first order algorithms such as low per iteration cost, scalability for very large problems, parallel and distributed implementations, and at the same time achieves the higher accuracy level of interior point methods. The major contributions and novelty of {this paper} are as follows

 \begin{enumerate}
 	\item We propose a highly scalable, simple iterative, and parallelizable first order algorithm for solving large conic optimization programs. 

    \item {We illustrate that a smaller condition number does not necessarily guarantee the faster convergence as the problem data matrices with a higher condition number can converge faster}. 
 	\item We propose a heuristic adaptive conditioning policy to obtain accurate solutions in comparison with  other first order algorithms and a proof is provided to guarantee the convergence of algorithm.
 	\item  We apply the proposed algorithm on {graphics processing unit (GPU)} to benefit the simple arithmetic operations in each iteration.
 	\item A wide range of tests are conducted on different conic programs and results are compared with several first order and interior point methods to justify the claims of scalability, efficiency and accuracy.   
 \end{enumerate}

The organization of the rest of this paper is as follows. Some preliminaries of cone programming and definitions are presented in section \ref{sec:pre}.  The effect of preconditioning and the need for proposed adaptive conditioning is illustrated in section \ref{sec:preconditioning} by  providing a numerical example. In section \ref{sec:adapt_cond}, we investigate the conditioning procedure to accelerate the convergence and provide an algorithm for adaptive conditioning. We compare the performance of proposed algorithm and adaptive conditioning in section \ref{sec:num_exp}, by solving a wide range of problems and comparing the results with commonly used solvers, and section \ref{sec:concl} concludes the paper. 

 \subsection{Notations}
Symbols $\Rbb$ and $\Nbb$ denote the set of real and natural numbers, respectively.  Matrices and vectors are represented by bold uppercase, and bold lowercase letters, respectively. Notation $\lVert \cdot {\rVert}_2$ refers to $\ell_2$ norm of either matrix or vector depending on the context and $\lvert \cdot\rvert$ represents the absolute value.   The symbol $\!(\cdot)^{\!\top}\!$ represent the  transpose operators. The notations $\boldsymbol{I}_n$  refer to the $n\times n$ identity  matrix.   The symbol $\Kcal$ is used to describe different types of cones used in this paper. The superscript $\!(\cdot)^{\!\mathrm{opt}}\!$ refers to the optimal solution of optimization problem. The notation $\!(\cdot)^{\dagger}$ denotes the  Moore–Penrose pseudoinverse of transpose of a matrix. The symbol $\Lcal$ represent the set of values to apply adaptive conditioning. The notation $\mathrm{diag}\{\cdots\}$ represent the diagonal elements of a diagonal matrix. The symbols $ \mathrm{SK}, \mathrm{AC}, \mathrm{CS}$, are used to refer Sinkhorn-Knopp, adaptive conditioning and competing solver, respectively. The symbols $\varepsilon^{\mathrm{abs}}$ and $\varepsilon^{\mathrm{rel}}$ are used for absolute and relative tolerance, respectively.

\section{Preliminaries  }\label{sec:pre}

In this paper, we consider the class of convex optimization problems with { a linear objective, } subject to a set of affine and second-order conic constraints. The primal formulation under study can be cast as:
\begin{subequations}\label{eq:prob_primal}
 	\begin{align}
  	& \underset{
  		\begin{subarray}{c} \!\!\!\! \!\!\! \! \xbf\in\,\Rbb^{n}
  		\end{subarray}
  	}{\text{minimize~~~}}
  	& &\hspace{-2cm} \cbf^{\top}\xbf  \label{eq:prob_obj}\\
  	& \text{subject to~~~}
  	& &\hspace{-2cm}  \Abf \xbf  = \bbf  \label{eq:prob_constraint}\\
  	& & & \hspace{-2cm} \xbf \in \Kcal  \label{eq:prob_cone}
  	\end{align}
\end{subequations}
where $\cbf \in \Rbb^{n},  \Abf \in \Rbb^{m\times n}$, and $ \bbf \in \Rbb^{m} $ are given and $\xbf\in\Rbb^n$ is the unknown optimization variable. Additionally, $\Kcal \triangleq\Kcal_{n_1}\times\Kcal_{n_2}\times\cdots\times\Kcal_{n_k}\subseteq\Rbb^n$, where each $\Kcal_{n_i}\subseteq\Rbb^{n_i}$ is a Lorentz cone of size $n_i$, i.e.,
\begin{align}
\Kcal_{n_i}\triangleq\big\{\wbf\in\Rbb^{n_i}\,|\,w_1\geq\big\|[w_2,\ldots,w_{n_i}]\big\|_2\big\},\nonumber
\end{align}
and $n_1+n_2+\ldots+n_k= n$.  
  
 The corresponding dual formulation of \eqref{eq:prob_primal} is 
\begin{subequations}\label{eq:prob_dual}
	\begin{align}
	& \underset{
		\begin{subarray}{c} \!\!\!\! \!\!\! \! \ybf \in \Rbb^m , \zbf \in \Rbb^n
		\end{subarray}
	}{\text{maximize~~~}}
	&&\hspace{-2cm} \bbf^{\top}\ybf  \label{eq:dual_obj}\\
	& \text{subject to~~~}
	&&\hspace{-2cm} \Abf^{\top}\ybf + \zbf = \cbf  \label{eq:dual_constraint}\\
	&&& \hspace{-2cm} \zbf  \in \Kcal \label{eq:dual_cone}
	\end{align}
\end{subequations}
where $\ybf$ and $ \zbf $ are dual variables associated with the constraints \eqref{eq:prob_constraint} and \eqref{eq:prob_cone}, respectively. 

 In this paper, we pursue a proximal numerical method inspired by Douglas-Rachford splitting \cite{EB92, GB17} to solve the class of optimization problems of the form \eqref{eq:prob_primal}. 
To this end, the projection and absolute value operators are defined as follows.   
\begin{definition}\label{def:abs}
	\vspace{1mm}
	For any proper cone $\Ccal \in \Rbb^n$, define the projection operator $\mathrm{proj}_{\Ccal}: \Rbb^n\to\Ccal$ as
	\begin{align*}
		\mathrm{proj}_{\Ccal}(\vbf)  \triangleq \mathrm{argmin}_{\ubf \in \Ccal} \;\;\|\ubf - \vbf\|_2.
	\end{align*}
	Additionally, define the absolute value operator $\mathrm{abs}_{\Ccal}: \Rbb^n\to\Rbb^n$  associated with $\Ccal$ as
	\begin{align*}
		\mathrm{abs}_{\Ccal}(\xbf_0) \triangleq 2 \mathrm{proj}_{\Ccal}(\xbf_0) -\xbf_0.
	\end{align*}
\end{definition}

\begin{algorithm}[t]
	\caption{~}
	\label{alg:1}
	\vspace{0.2mm}
	\begin{algorithmic}[1]
		\Require{ $(\Abf,\bbf,\cbf,\Kcal)$, fixed $\mu > 0$, and initial point $\sbf\in\Rbb^n$}
		\State $\Acal:=\mathrm{range}\{\Abf^{\!\top}\}\phantom{\Big|}$
		\State $\dbf := \Abf^{\dagger}\bbf + \dfrac{\mu}{2}\left(\mathrm{abs}_{\Acal}(\cbf)-\cbf\right) \phantom{\big|}$
		\Repeat
		\State $ \pbf \gets  \mathrm{abs}_{\Kcal}(\sbf)$
		\State $ \rbf \gets  \mathrm{abs}_{\Acal}(\pbf)\phantom{\Big|}$
		\State $\sbf \gets \dfrac{\sbf}{2} - \dfrac{\rbf}{2} + \dbf\phantom{\Big|}$
		\Until {stopping criteria is met.$\phantom{\Big|}$}
		\Ensure \!
		$\xbf
		\!\gets\!\dfrac{\pbf\!+\!\sbf}{2}$,\;\;
		$\zbf
		\!\gets\!\dfrac{\pbf\!-\!\sbf}{2\mu}$
	\end{algorithmic}\label{al:alg_1}
\end{algorithm}

Algorithm \ref{al:alg_1} details the proposed first-order numerical method for solving \eqref{eq:prob_primal}. 

\begin{theorem}\label{thm1}
Let $\{\sbf^l\}^{\infty}_{l=0}$ and $\{\pbf^l\}^{\infty}_{l=0}$ denote the sequence of vectors generated by Algorithm \ref{al:alg_1}. Then we have
\begin{align}
\lim_{l\to\infty} \dfrac{\pbf^l+\sbf^l}{2} = \bar{\xbf}\qquad\mathrm{and}\qquad
\lim_{l\to\infty} \dfrac{\pbf^l-\sbf^l}{2\mu} = \bar{\zbf}
\end{align}
where $\bar{\xbf}$ and $\bar{\zbf}$ are a pair of primal and dual solutions for problems \eqref{eq:prob_primal} and \eqref{eq:prob_dual}, respectively.
\end{theorem}
\begin{proof}
Please see the Appendix for the proof.
\end{proof}

Despite its potential for massive parallelization {for both CPU and GPU architectures}, in and of itself, Algorithm \eqref{alg:1} may not offer any advantages over the common-practice Douglas-Rachford splitting (DR) and the Alternating Direction Method of Multipliers (ADMM). However, as we will demonstrate next, Algorithm \eqref{alg:1} enables us to perform adaptive conditioning to achieve much faster convergence speed in comparison with the state-of-the-art pre-conditioning methods.

\section{State of the Art Preconditioning Methods}\label{sec:preconditioning}

One of the major drawbacks of first-order numerical methods is their sensitivity to the problem conditioning \cite{GB2015, FB2018,SBGB2018 }. Hence, it is common practice to reformulate problem \eqref{eq:prob_primal} with respect to new parameters
\begin{align}
\hat{\Abf}\triangleq\Dbf\Abf\Ebf,\quad
\hat{\bbf}\triangleq\Dbf\bbf,\quad \mathrm{and}\quad
\hat{\cbf}\triangleq\Ebf\cbf
\end{align}
and new proxy variables
\begin{align}
\hat{\xbf}=\Ebf^{-1}\xbf \quad\mathrm{and}\quad \hat{\zbf}=\Ebf^{\top}\zbf 
\end{align}
where $\Dbf \in \Rbb^{m \times m}$ and $\Ebf \in \Rbb^{n\times n}$ are tuned to improve convergence speed. The process of finding an appropriate $\Dbf$ and $\Ebf$ to improve the performance of a first-order numerical algorithm is regarded as preconditioning of data. 

Theoretical and practical evidence show that choices of $\Dbf$ and $\Ebf$ that result in smaller condition number for $\hat{\Abf}$ lead to better performance in both precision and convergence rate of first-order numerical algorithms \cite{GB14,GB2014a,GB2015,GTS2015}. As a result, over the past decade, several research directions have pursued preconditioning methods such as heuristic diagonal scaling with the aim of reducing the condition number of  $\hat{\Abf}$ \cite{GB14,PC11}. To this end, a number of  matrix equilibration heuristics  such as Sinkhorn-Knopp and Ruiz methods have been proposed in \cite{FB2018,ZFP+20,DB2017,SBGB2018} that indirectly influence the condition number of $\hat{\Abf}$ by equalizing $\ell_p$ norm for each row through diagonal choices of $\Dbf$ and $\Ebf$.

In this paper, we pursue an alternative approach. We argue that the condition number of $\hat{\Abf}$ is not a reliable indicator of convergence speed for first-order numerical methods and instead, we offer a new approach regarded as {\it adaptive {conditioning}}. Before elaborating the details of the proposed {procedure}, we first give a simple illustrative example through which it is shown that {a} smaller condition number for the data matrix $\hat{\Abf}$ does not necessarily result in better performance. 


\begin{figure*}[t]
	\subfloat[\label{fig:dr}]{\includegraphics[width= 0.333\textwidth]{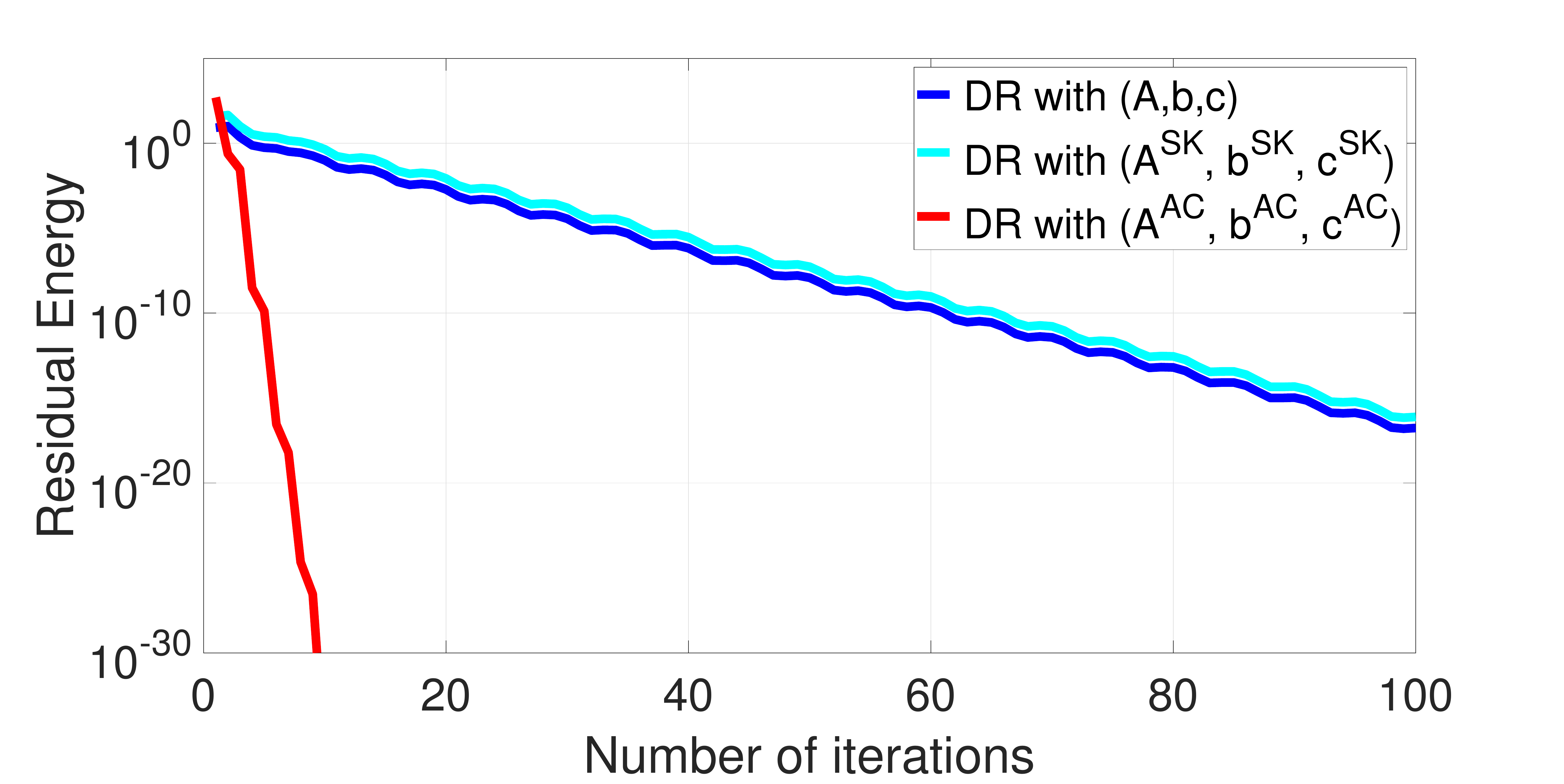}}	
	\subfloat[\label{fig:admm}]{\includegraphics[width= 0.333\textwidth]{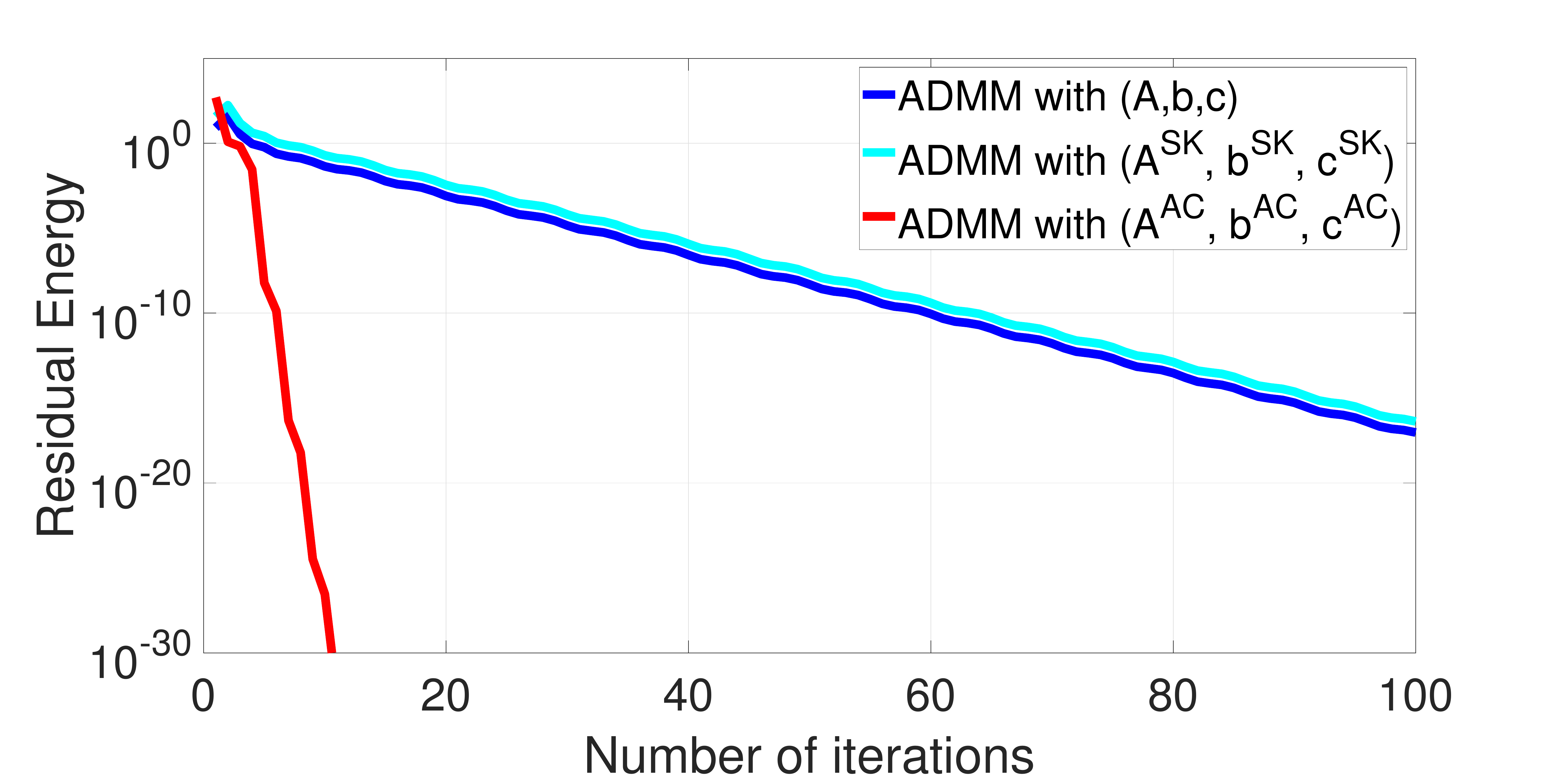}}
	\subfloat[\label{fig:alg1}]{\includegraphics[width= 0.333\textwidth]{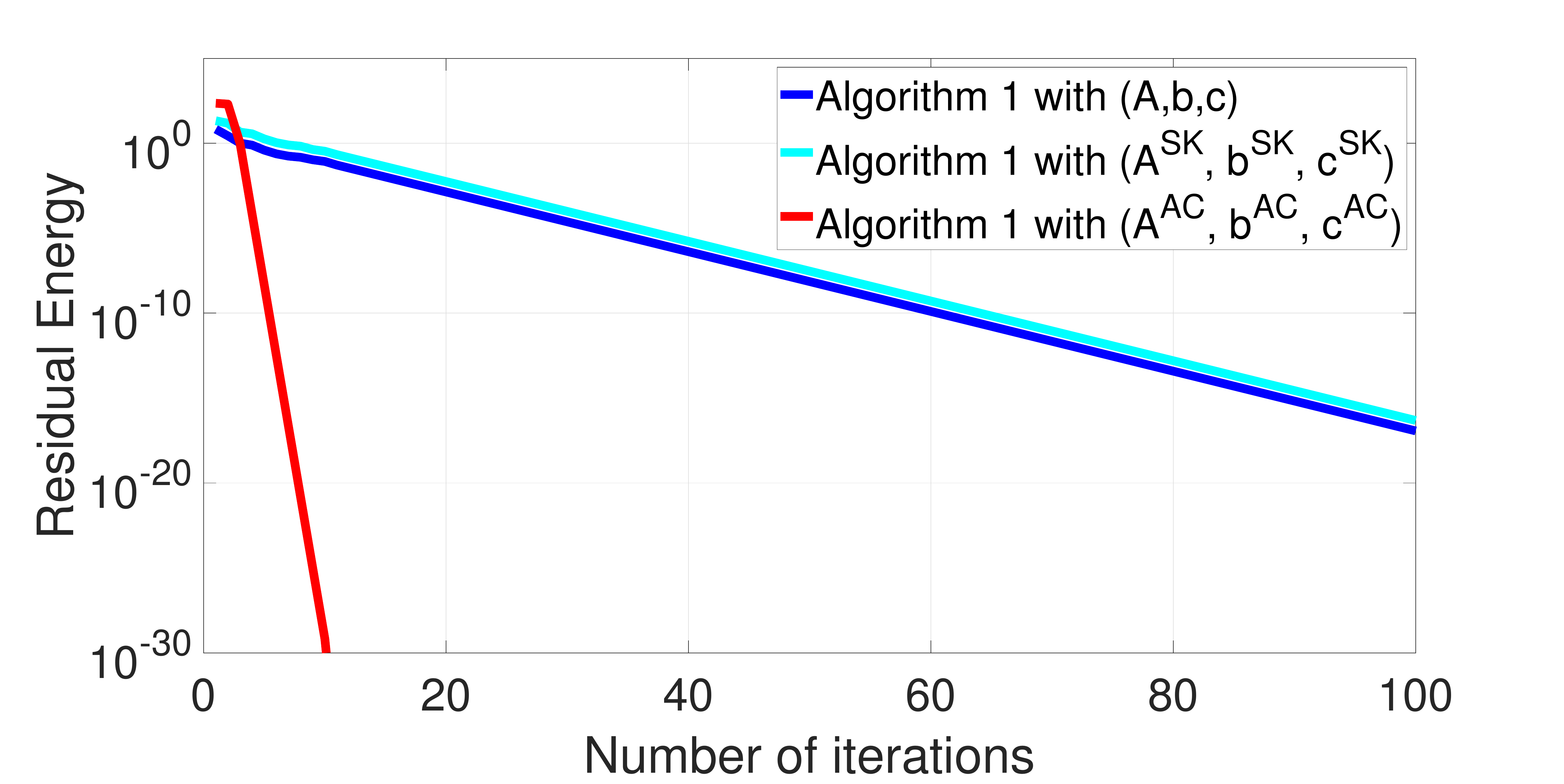}}
	\caption{The effect of different pre-conditioning methods on the convergence of (a) Douglas-Rachford splitting, (b) Alternating Direction Method of Multipliers, and (c) Algorithm \ref{al:alg_1}.}
	\label{fig:algorithm_comparison}
\end{figure*}


\subsection{Example: The effect of condition number}
In this example, we provide simple data matrices and compare the effect of different pre-conditioning methods on the convergence of Algorithm \eqref{al:alg_1}, DR splitting, and ADMM. The goal is to demonstrate {that} the condition number of $\hat{\Abf}$ is not a reliable predictor of {the} convergence speed. 

Consider the following data matrices:
\begin{subequations}
\begin{align*} 
	\Abf &:= \begin{bmatrix} 3.57  &  3.45 &   3.33 &   64.24 &    -72.76 \\
		3.45  &  3.33 &   3.23 &   95.14 &    -23.34 \\
		3.33  &  3.23 &   3.13 &   93.53 &    -17.43 \\
	\end{bmatrix},\\
	\bbf &:= \begin{bmatrix}
		-10.44 &20.65 &22.94
	\end{bmatrix}^{\top}, \\
	\cbf &:= \begin{bmatrix}
		0.37 &1.93 & -0.12 &-0.38 & 1.01
	\end{bmatrix}^{\top},
\end{align*}
\end{subequations}
and $\Kcal:=\Rbb_+^5$. The corresponding diagonal matrices obtained from regularized Sinkhorn-Knopp algorithm \cite{FB2018} for $\ell_2$ {norms} are
\begin{subequations}
\begin{align*}\small
&\Dbf^{\mathrm{SK}}=\mathrm{diag}\{[0.0217, 0.0215, 0.0222]\},\\
&\Ebf^{\mathrm{SK}}=\mathrm{diag}\{[0.4722, 0.4722, 0.4722, 0.4722, 0.4722]\},
\end{align*}
\end{subequations}
while the proposed heuristic adaptive conditioning results in the following matrices
\begin{subequations}
	\begin{align*}\small
		&\Dbf^{\mathrm{AC}}\!=\!\Ibf_{3\times 3},\\
		&\Ebf^{\mathrm{AC}}\!=\!\mathrm{diag}\{[0.0792, 0.0884, 14.5484, 292.9524, 316.2179]\}.
	\end{align*}
\end{subequations}
Define
\begin{align*}
\hat{\Abf}^{\mathrm{SK}}:=\Dbf^{\mathrm{SK}}\Abf \Ebf^{\mathrm{SK}}\quad\mathrm{and}\quad
\hat{\Abf}^{\mathrm{AC}}:=\Dbf^{\mathrm{AC}}\Abf \Ebf^{\mathrm{AC}}.
\end{align*}
In this case, the condition numbers of  $\Abf$, $\hat{\Abf}^{\mathrm{SK}}$, and $\hat{\Abf}^{\mathrm{AC}}$ are equal to $2046.4$, $2044.38$, and $72079.13$, respectively.

\subsubsection{Douglas-Rachford Splitting}
In order to implement the DR splitting method, it is common practice to cast problem \eqref{eq:prob_primal} in the form of
\begin{align}
	&{\text{minimize}}
	&&\hspace{-2cm} f(\xbf) + g(\xbf)  \label{eq:split_opt}
\end{align}
where $f,g:\Rbb^n\to\Rbb\cup\{\infty\}$ 
are defined as
\begin{align}
	f(\xbf)\! \triangleq\!  \left\{\begin{matrix} 0 & \!\!\mathrm{if }\  \xbf \in \Kcal \\
		\infty & \!\!\mathrm{ \ otherwise} \end{matrix}\right.  
	\quad\mathrm{and}\quad
	g(\xbf) \!\triangleq\!  \left\{\begin{matrix} \cbf^{\top}\xbf & \!\!\mathrm{ if }\ \Abf\xbf=\bbf \\ \infty & \!\!\mathrm{otherwise} \end{matrix}\right.  \nonumber 
\end{align}
leading to the following steps:
\begin{subequations}
\begin{align}
	&\xbf\gets\mathrm{prox}_f(\zbf)\\
	&\zbf\gets\zbf+\mathrm{prox}_g(2\xbf-\zbf)-\xbf.
\end{align}
\end{subequations}

\subsubsection{Alternating Direction Method of Multipliers} A standard way of solving problem \eqref{eq:prob_primal} via ADMM is through the formulation
\begin{subequations}
\begin{align}
& \underset{
	\begin{subarray}{c} \!\!\!\! \!\!\! \! \xbf_1,\xbf_2 \in \Rbb^n
	\end{subarray}
}{\text{minimize~~~}}
&&\hspace{-2cm} f(\xbf_1) + g(\xbf_2) \\
& \text{subject to~~~}
&&\hspace{-2cm} \xbf_1=\xbf_2
\end{align}
\end{subequations}
leading to the steps
\begin{subequations}
\begin{align}
	&\xbf_1 \gets \mathrm{prox}_{\mu^{-1}f}(\xbf_2-\mu^{-1}\zbf)\\
	&\xbf_2 \gets \mathrm{prox}_{\mu^{-1}g}(\xbf_1+\mu^{-1}\zbf)\\
	&\zbf\gets\zbf+\mu(\xbf_1-\xbf_2).
\end{align}
\end{subequations}
where $\mu$ is a fixed tuning parameter.

Figure \eqref{fig:algorithm_comparison} presents the outcome of DR splitting, ADMM, and Algorithm \eqref{al:alg_1}, respectively, with $\mu=1$ and different pre-conditioning methods. The three cases of no preconditioning, Sinkhorn-Knopp preconditioning, and adaptive conditioning are illustrated in each figure. As demonstrated in Figure \ref{fig:algorithm_comparison}, a lower condition number for the data matrix does not necessarily result in a faster convergence. Motivated by this observation, the following section presents the proposed adaptive conditioning procedure.

\begin{algorithm}[t]
	\caption{~}
	\label{alg:O}
	\vspace{0.2mm}
	\begin{algorithmic}[1]
		\Require $(\Abf,\bbf,\cbf,\Kcal)$, fixed $\mu\!>\!0$, initial points $\xbf\in\Rbb^n$ and $\zbf\in(\Rbb\setminus\{0\})^n$, fixed $0<t<1$, and $\Lcal\subseteq\Nbb$
		\State $l \gets 0$
		\Repeat
		\State $l \gets l+1$
		\If {$l\in\Lcal\cup\{1\}$}
		\For{$i=1,\ldots,k$}
		\State $h\gets n_1+\ldots+n_{i-1}$
		\For{$j=h+1,\ldots,h+n_{i}$}
		\Statex\vspace{-3.5mm}
		\State $\obf_{j}\!\gets\! \big|x_{h+1}\!-\!\|[x_{h+2},\ldots,x_{h+n_i}]\|_2\big|/|z_{h+1}|$
		\Statex\vspace{-3.5mm}
		\EndFor
		\EndFor
		\State $\Ob \gets \mathrm{diag}\{|\obf|\}^{\min\left\{1,\frac{t}
			{\log(\max\{\obf\})-\log(\min\{\obf\})}\right\}}$
		\State $\hat{\Acal}\gets\mathrm{range}\{(\Abf\Ob)^{\!\top}\}\phantom{\Big|}$
		\Statex\vspace{-2.0mm}
		\State $\dbf \gets (\Abf\Ob)^{\dagger}\bbf +\dfrac{\mu}{2}\left(\mathrm{abs}_{\Ob^{\top}\Acal}(\Ob^{\top}\cbf)-\Ob^{\top}\cbf\right) \phantom{\big|}$
		\Statex\vspace{-2.0mm}
		\State $\sbf \gets \Ob^{-1}\xbf-\mu\Ob\zbf$
		\EndIf
		\State $ \pbf \gets  \mathrm{abs}_{\Ob^{-1}\Kcal}(\sbf)$
		\State $ \rbf \gets  \mathrm{abs}_{\hat{\Acal}}(\pbf)\phantom{\Big|}$
		\State $\sbf \gets \dfrac{\sbf}{2} - \dfrac{\rbf}{2} + \dbf\phantom{\Big|}$
		\Statex
		\State $\xbf\!\gets\!\dfrac{\Ob(\pbf\!+\!\sbf)}{2}$
		\Statex
		\State $\zbf\!\gets\!\dfrac{\Ob^{-1}(\pbf\!-\!\sbf)}{2\mu}$
		\Until {stopping criteria is met.$\phantom{\Big|}$}
		\Ensure \!
$
\xbf$\;\; and\;\;
$
\zbf$
	\end{algorithmic}\label{al:alg_O}
\end{algorithm}

\begin{figure*}[t]
	\subfloat[\label{fig:no_cond}]{\includegraphics[width= 0.333\textwidth]{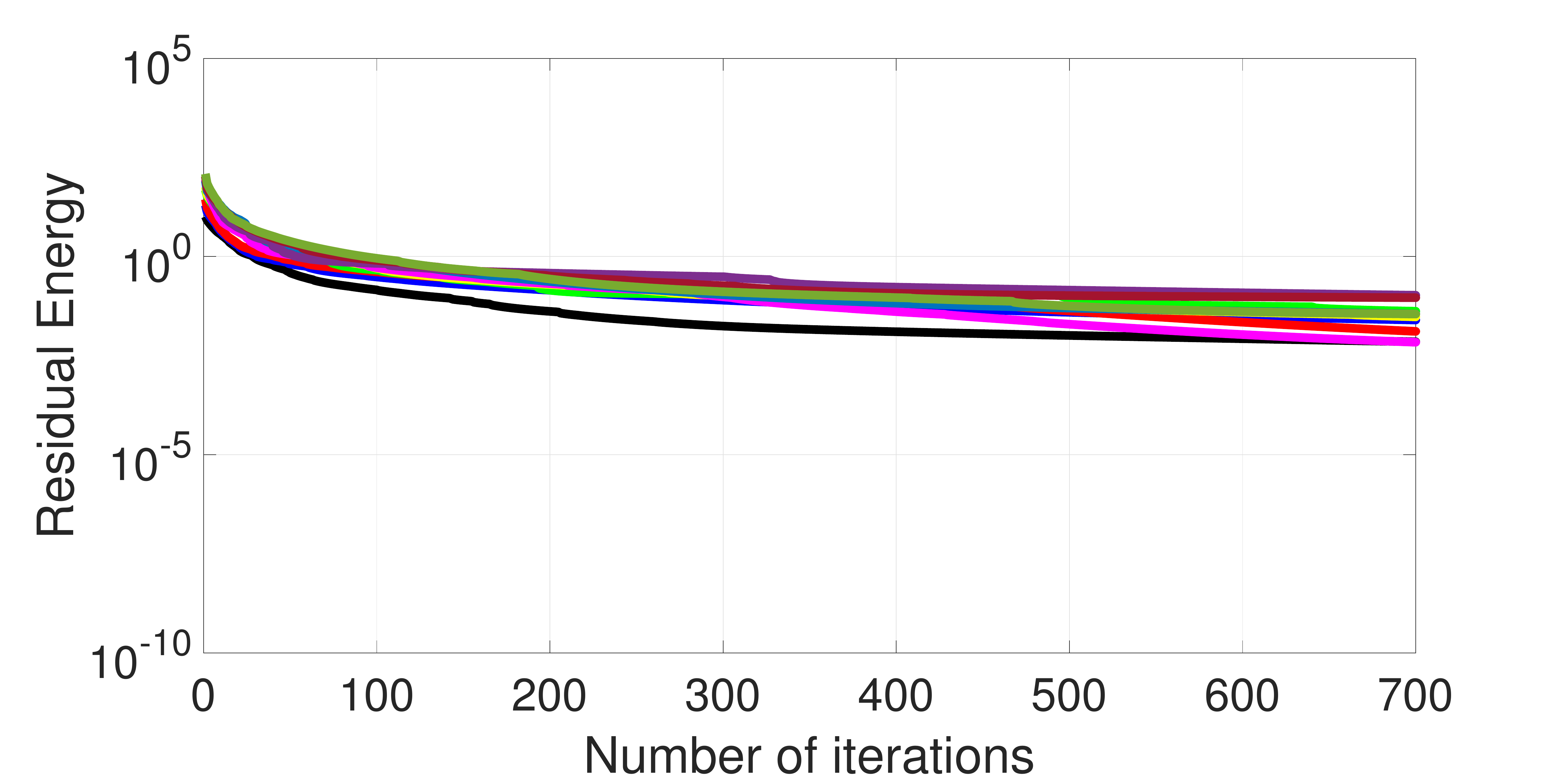}}
	\subfloat[\label{fig:one_cond}]{\includegraphics[width= 0.333\textwidth]{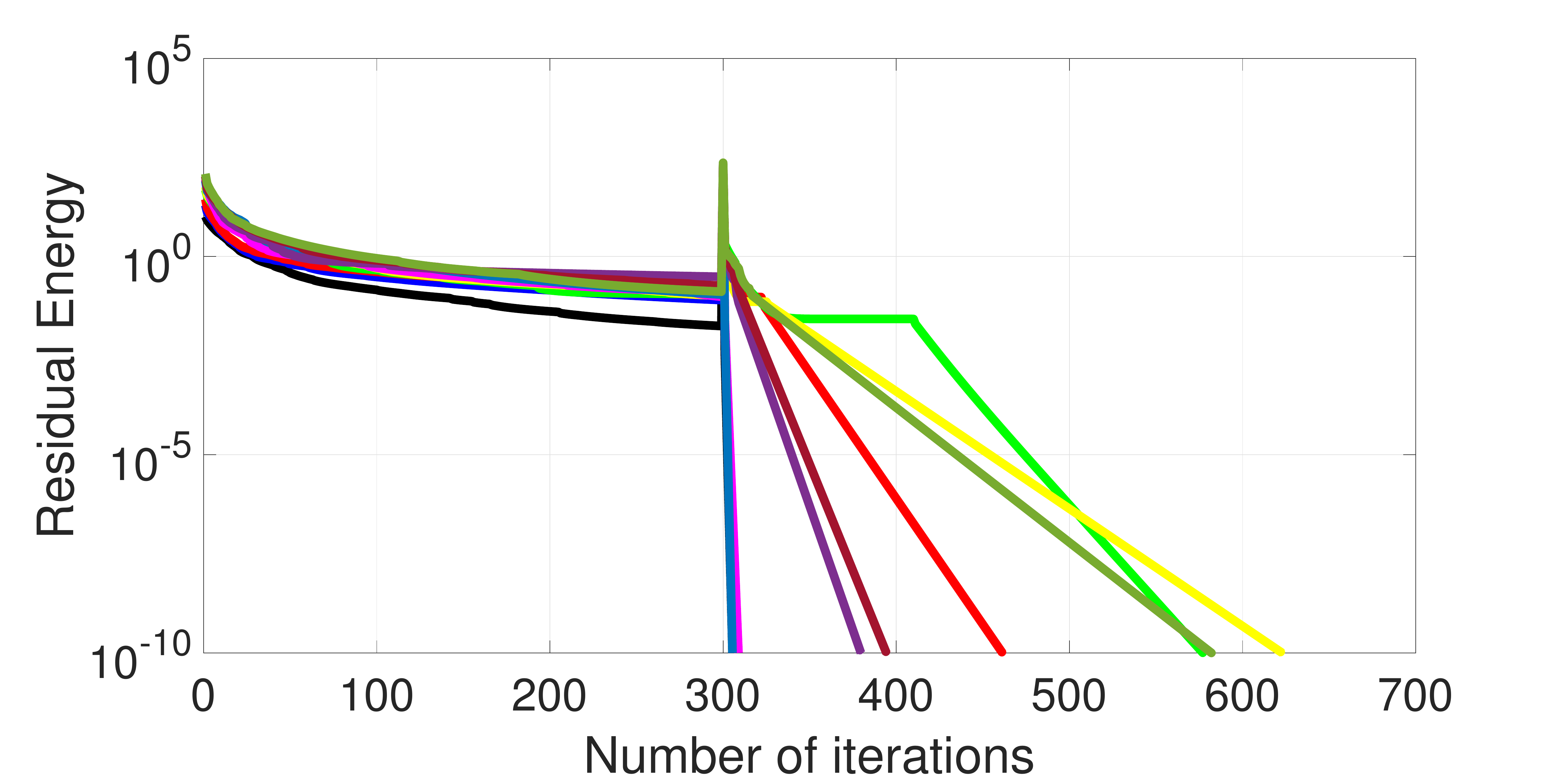}}
	\subfloat[\label{fig:multiple_cond}]{\includegraphics[width= 0.333\textwidth]{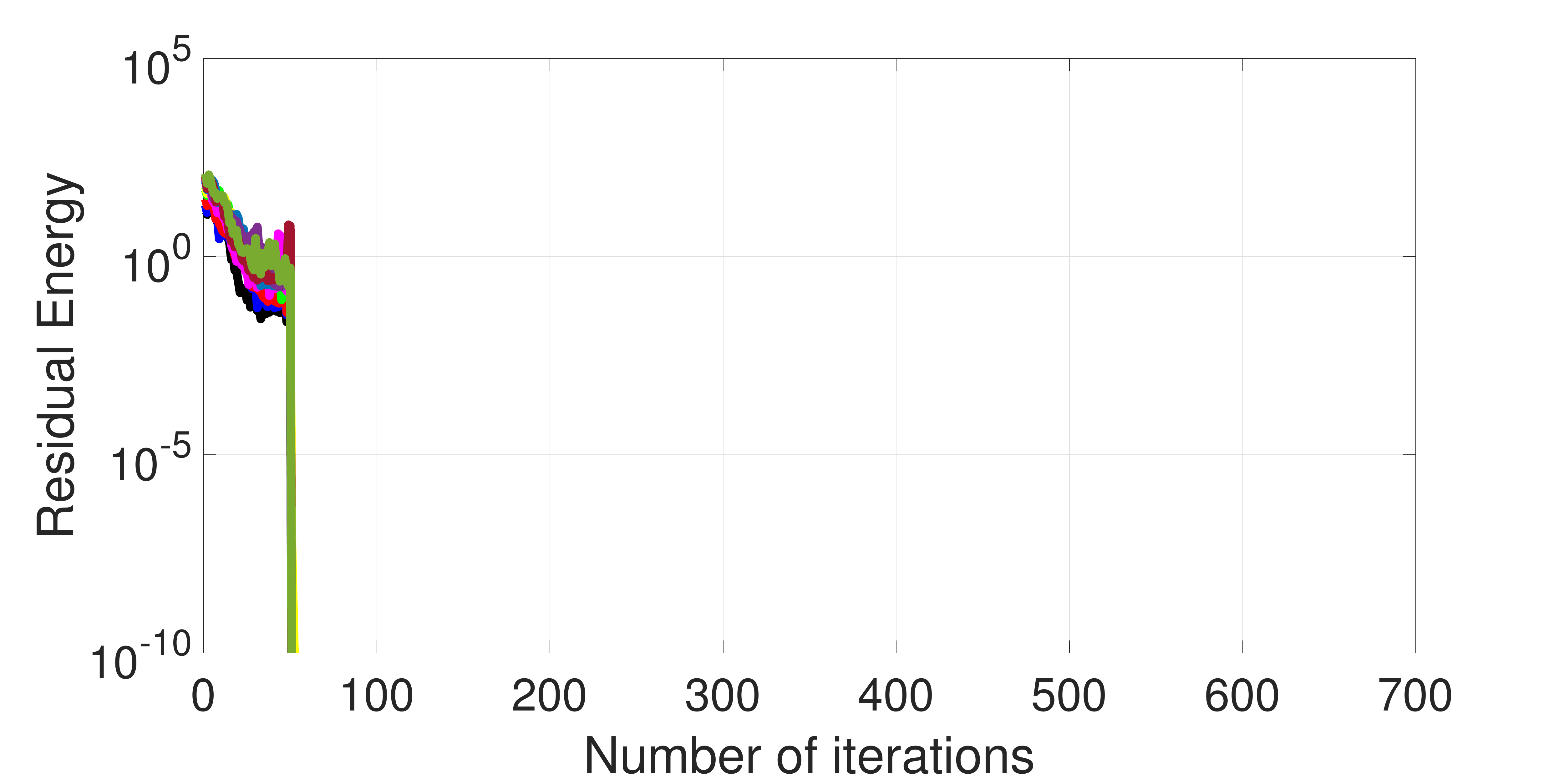}}
	\caption{Convergence of Algorithm \eqref{al:alg_O} for 10 random linear programming instance {(distinct color for each instance)} with different conditioning steps:
		(a) No conditioning, i.e., $\Lcal:=\varnothing$,
		(b) One time conditioning at iteration $l=300$, i.e., $\Lcal:=\{300\}$, and
		(c) {Continuous} conditioning at the first $50$ iterations, i.e., $\Lcal\!:=\!\{1,2,\ldots,50\}$.}
	\label{fig:lp_scaling}
\end{figure*}

\section{Adaptive Conditioning}\label{sec:adapt_cond}
In this work, we rely on post multiplication of the data matrix $\Abf$ by a diagonal positive-definite matrix $\Ob$, to enhance the convergence speed of Algorithm \ref{al:alg_1}. The primal problem \eqref{eq:prob_primal} is reformulated as:
\begin{subequations}
	\begin{align}
		& \underset{
			\begin{subarray}{c} \!\!\!\! \!\!\! \! \hat{\xbf}\in\,\Rbb^{n}
			\end{subarray}
		}{\text{minimize~~~}}
		& &\hspace{-2cm} (\Ob^{\top}\cbf)^{\top}\hat{\xbf}  \\
		& \text{subject to~~~}
		& &\hspace{-2cm}  (\Abf\Ob) \hat{\xbf}  = \bbf  \\
		& & & \hspace{-2cm} \hat{\xbf} \in \Ob^{-1}\Kcal  
	\end{align}
\end{subequations}
and the dual problem \eqref{eq:prob_dual} as:
\begin{subequations}
	\begin{align}
		& \underset{
			\begin{subarray}{c} \!\!\!\! \!\!\! \! \ybf \in \Rbb^m , \hat{\zbf} \in \Rbb^n
			\end{subarray}
		}{\text{maximize~~~}}
		&&\hspace{-1cm} \bbf^{\top}\ybf  \\
		& \text{subject to~~~}
		&&\hspace{-1cm} (\Abf\Ob)^{\top}\ybf + \hat{\zbf} = \Ob^{\top}\cbf  \\
		&&& \hspace{-1cm} \hat{\zbf} \in \Ob^{\top}\Kcal^{\ast}
	\end{align}
\end{subequations}  
where
\begin{align}
	&\hat{\xbf}\triangleq\Ob^{-1}\xbf\quad\mathrm{and}\quad
	\hat{\zbf}\triangleq\Ob^{\top}\zbf
\end{align}
are proxy variables.

In contrary to the existing practice that focuses on the condition number of the data matrix, we continuously update the matrix $\Ob$ according to a prespecified policy to improve the convergence speed. This heuristic procedure is detailed in Algorithm \ref{al:alg_O}. As illustrated in Figure \ref{fig:illustration}, the intuitive reason behind the proposed adaptive conditioning is to equalize the rate of convergence for elements of

\begin{itemize}
	\item {\bf Step 4:} Adaptive conditioning can be done based on a user-defined criteria or in the simplest case, at a set of user-defined iterations $\Lcal$.
	\item {\bf Step 5 and 11:} New coefficients are calculated for each cone to equalize the speed of convergence for the elements of $\xbf$ and $\zbf$. Note that since the vectors $\xbf$ and $\zbf$ are complementary at optimality, 
the elements of $\obf$ can be very large or very small numbers and that is the motivation behind the normalization step 11.
\item {\bf Steps 12 and 13:} These two steps are concerned with the adjustments of the proximal operators and the vector $\dbf$, respectively.
\item {\bf Step 14:} This step casts the vector $\sbf$ into the new space so that current progress is continued.
\end{itemize}

\begin{figure}[t]
	\includegraphics[width= \columnwidth]{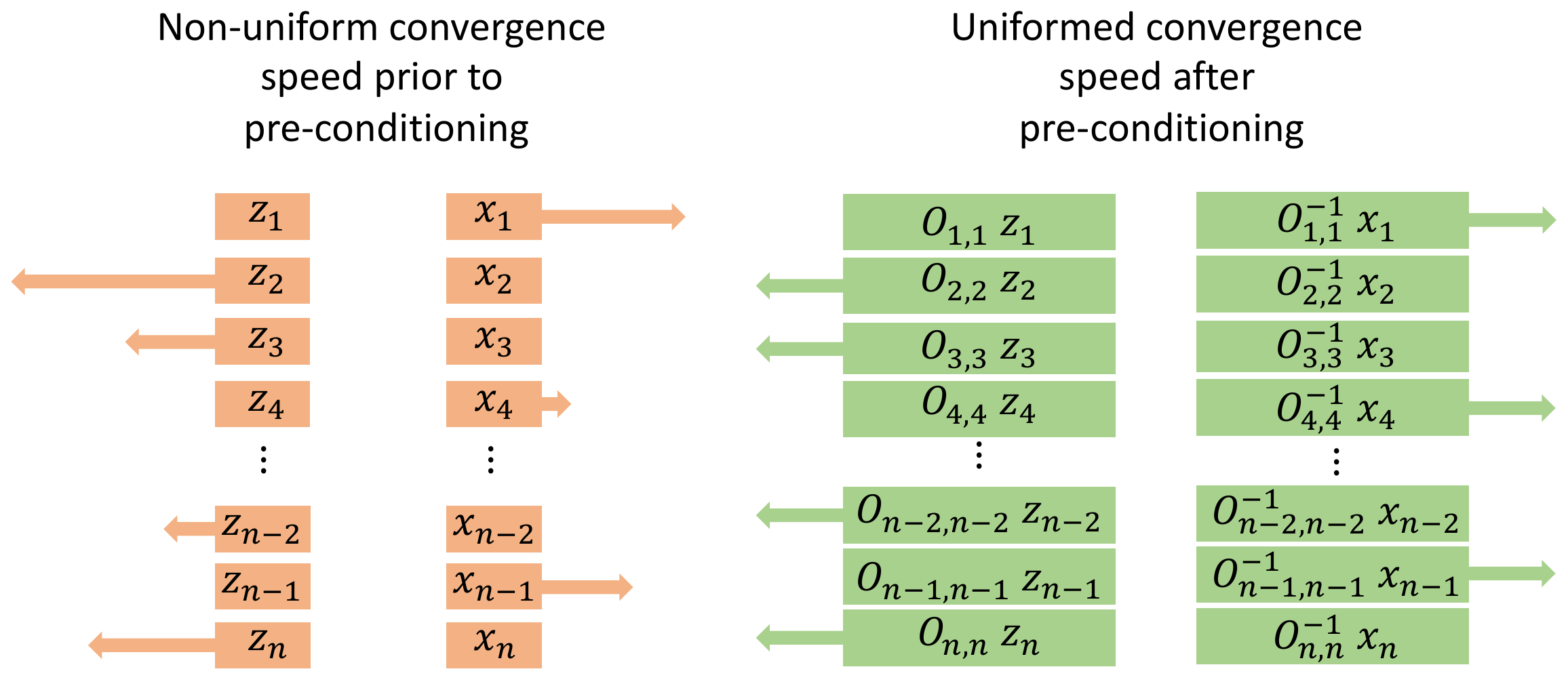}
	\caption{Intuitive reason behind adaptive conditioning { to equalize the convergence speed}}
	\label{fig:illustration}
\end{figure}

The next example demonstrates the effectiveness of the proposed adaptive conditioning approach on random instances of linear programming (LP).

\subsection{Example: The choice of conditioning steps}
This case study is concerned with the effect of conditioning steps on the convergence behavior of Algorithm \eqref{al:alg_O}. We consider three cases:
\begin{itemize}
\item No conditioning, i.e., $\Lcal:=\varnothing$,
\item One time conditioning at iteration $l=300$, i.e., $\Lcal:=\{300\}$,
\item {Continuous} conditioning at the first $50$ iterations, i.e., $\Lcal\!:=\!\{1,2,\ldots,50\}$.
\end{itemize}
We generated $10$ random instances of linear programming with $100$ variables and $80$ linear constraints  whose data are chosen such that:
\begin{itemize}
\item The elements of $\Abf\in\Rbb^{80\times 100}$ have i.i.d standard normal distribution.
\item $\bbf: =\Abf\dot{\xbf}$ where the elements of $\dot{\xbf}\in\Rbb^{100}$ have i.i.d uniform distribution from the interval $[0,1]$.
\item The elements of $\cbf\in\Rbb^{100}$ have i.i.d standard normal distribution.
\item And $\Kcal=\Rbb_+^{100}$.
\end{itemize}
The effect of adaptive conditioning proposed in Algorithm \ref{al:alg_O} for $t = 9.2$ is illustrated in Figure \ref{fig:lp_scaling} for all 10 random instance. As demonstrated in the figure, even a one time adaptive conditioning results in significant improvement of convergence speed.  

\section{Numerical Experiments}\label{sec:num_exp}
In this section we provide case studies to evaluate the performance of Algorithm \ref{alg:O} on both CPU and GPU platforms in comparison with the state-of-the-art commercial solvers MOSEK \cite{mosek}, GUROBI  \cite{gurobi} as well as the open source software OSQP  \cite{SBGB2018} and POGS \cite{FB2018}. Our case studies consist of randomly generated linear programming (LP) and second-order cone programming (SOCP) problems. We conduct experiments on problems with a wide range of variable and constraint numbers {to} assess both scalability and speed. Additionally, we consider different values for infeasiblity/gap tolerance, { to } assess the solution accuracy of Algorithm \ref{alg:O}. 
{The proposed algorithm and competing solvers are implemented in MATLABR2020a and all the simulations are conducted on a }  DGX station with 20 2.2 GHz cores, Intel Xeon E5-2698 v4 CPU, with NVIDIA Tesla V100-DGXS-32GB (128 GB total) GPU processor and 256 GB of RAM. { The parallel nature of algorithm enables the implementation to take advantage of multi-core CPU processing. Note that our implementation of proposed algorithm in MATLAB utilizes only a single GPU and does not benefit from multiple GPU's of the platform. Moreover, all experiments reported in this paper are not bounded by RAM or GPU memory of DGX station.} We used the MATLAB interface of OSQP v0.6.0, MOSEK v9.2.5 and GUROBI v9.0.

In all of the experiments, the stopping criteria of Algorithm \ref{alg:O} is when it {exceeds} both primal and dual feasibility of the solution produced by the competing solver. In other words, when the following two criteria are met:  
\begin{subequations}\label{stopcri}
\begin{align}
&\|\Abf \xbf-\bbf\|_2 < \|\Abf \xbf^{\text{CS}}-\bbf\|_2\\
&|(\Abf^{\dagger}\bbf)^{\top}(\cbf-\zbf) - \cbf^{\top}\xbf|<|\bbf^{\top}\ybf^{\text{CS}} - \cbf^{\top}\xbf^{\text{CS}}|
\end{align}
\end{subequations}
where $\xbf^{\text{CS}}$ and $\ybf^{\text{CS}}$ are primal and dual solutions produced by the competing solver under default settings. In each figure, the experiments are continued until the run time of the competing solver reached a maximum time of 1200 seconds. { The maximum time is chosen in such a way that the experiments provide sufficient information to compare the computational time for all solvers. }

\begin{figure*}[t] 
	\subfloat[\label{fig:lp_osqp}] {\includegraphics[width =0.333\textwidth]{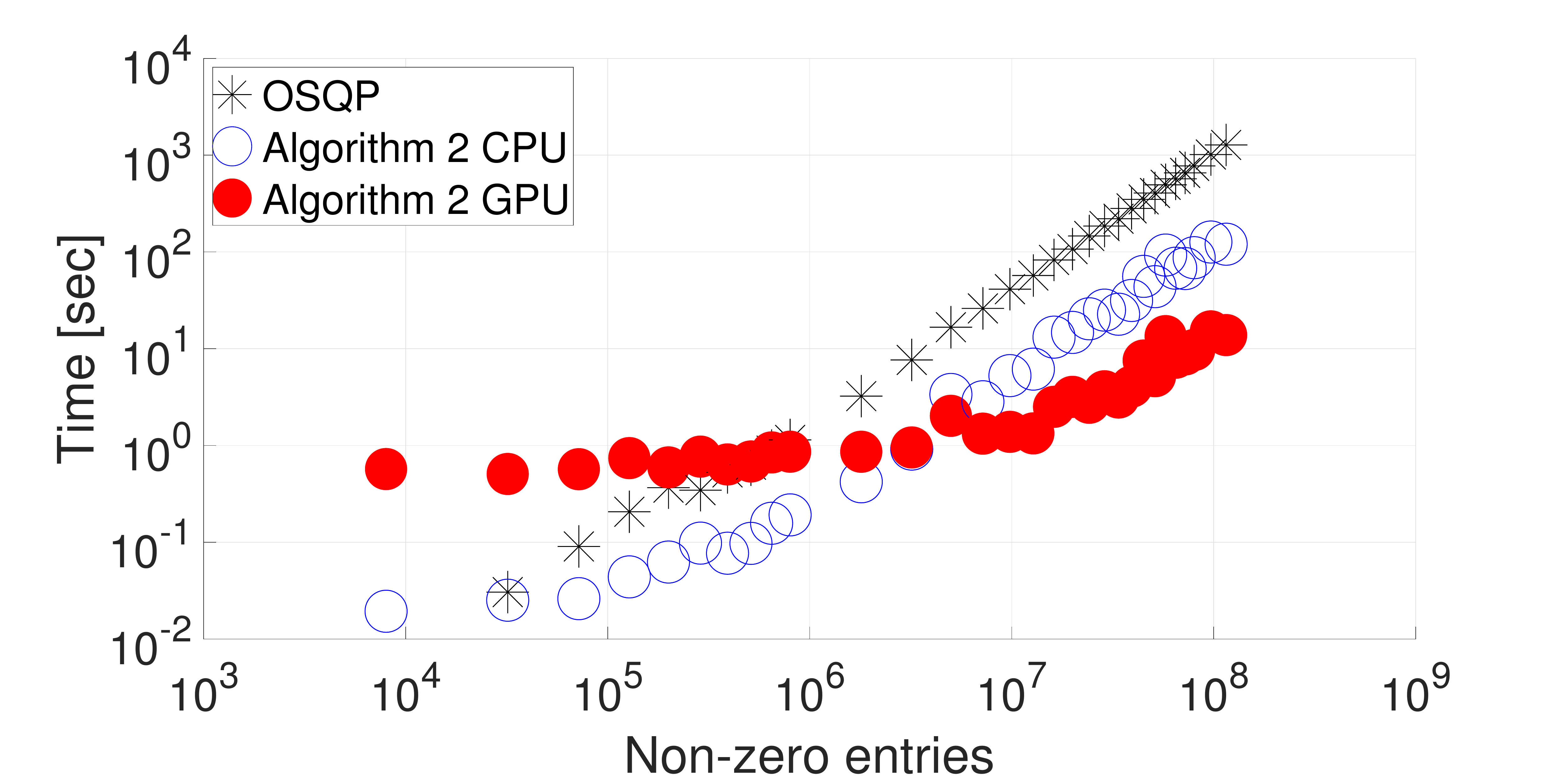}}
	\subfloat[\label{fig:lp_gurobi}] {\includegraphics[width =0.333\textwidth]{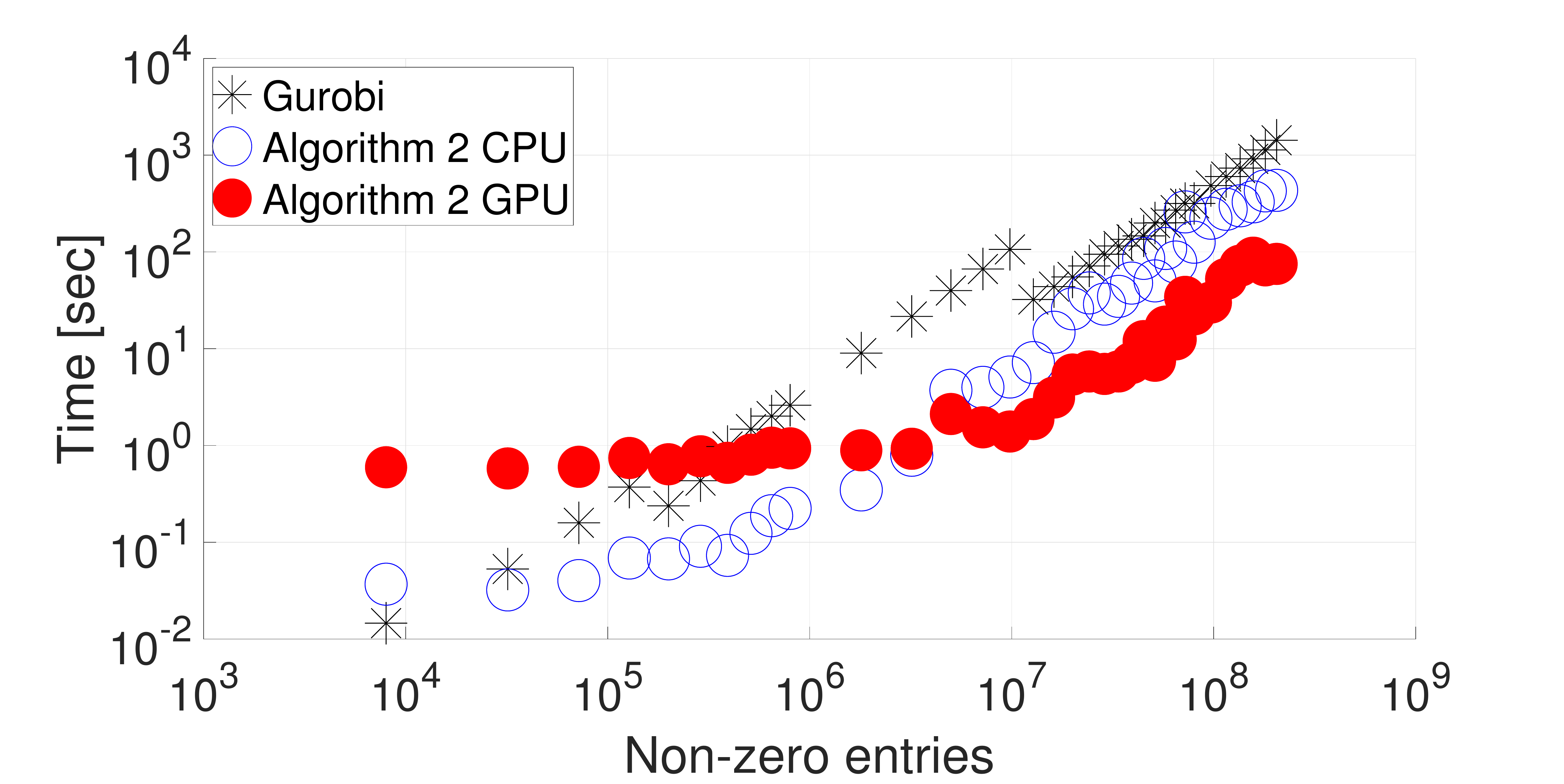}}
	\subfloat[\label{fig:lp_mosek}] {\includegraphics[width =0.333\textwidth]{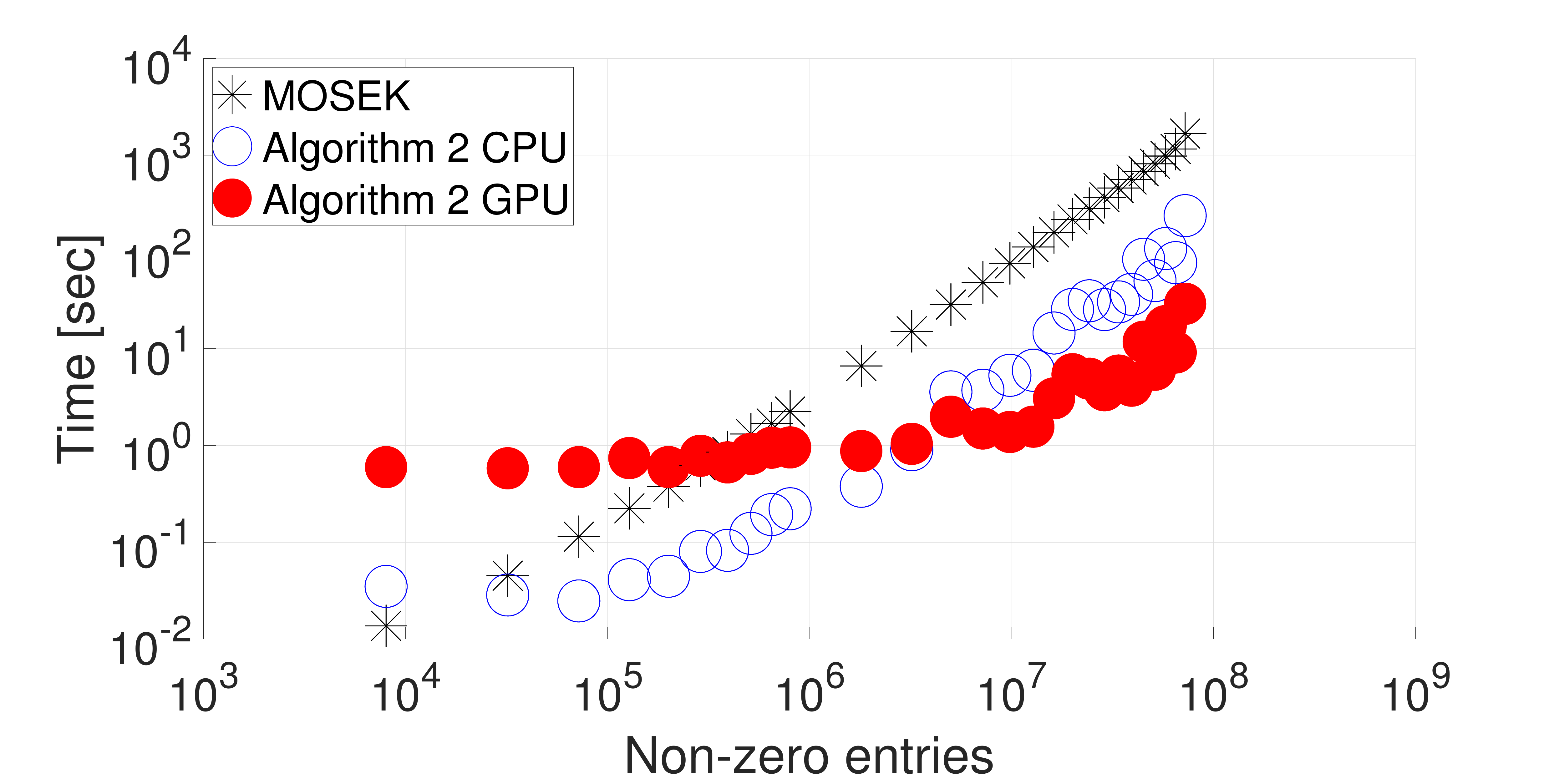}}
	\caption{The performance of Algorithm \ref{al:alg_O} for linear programming in comparison with (a) OSQP,  (b) GUROBI, and (c) MOSEK.}
	\label{fig:lp_computation_time}
\end{figure*}

\begin{figure*}[t] 
	\subfloat[\label{fig:lp_pogs1}] {\includegraphics[width =0.333\textwidth]{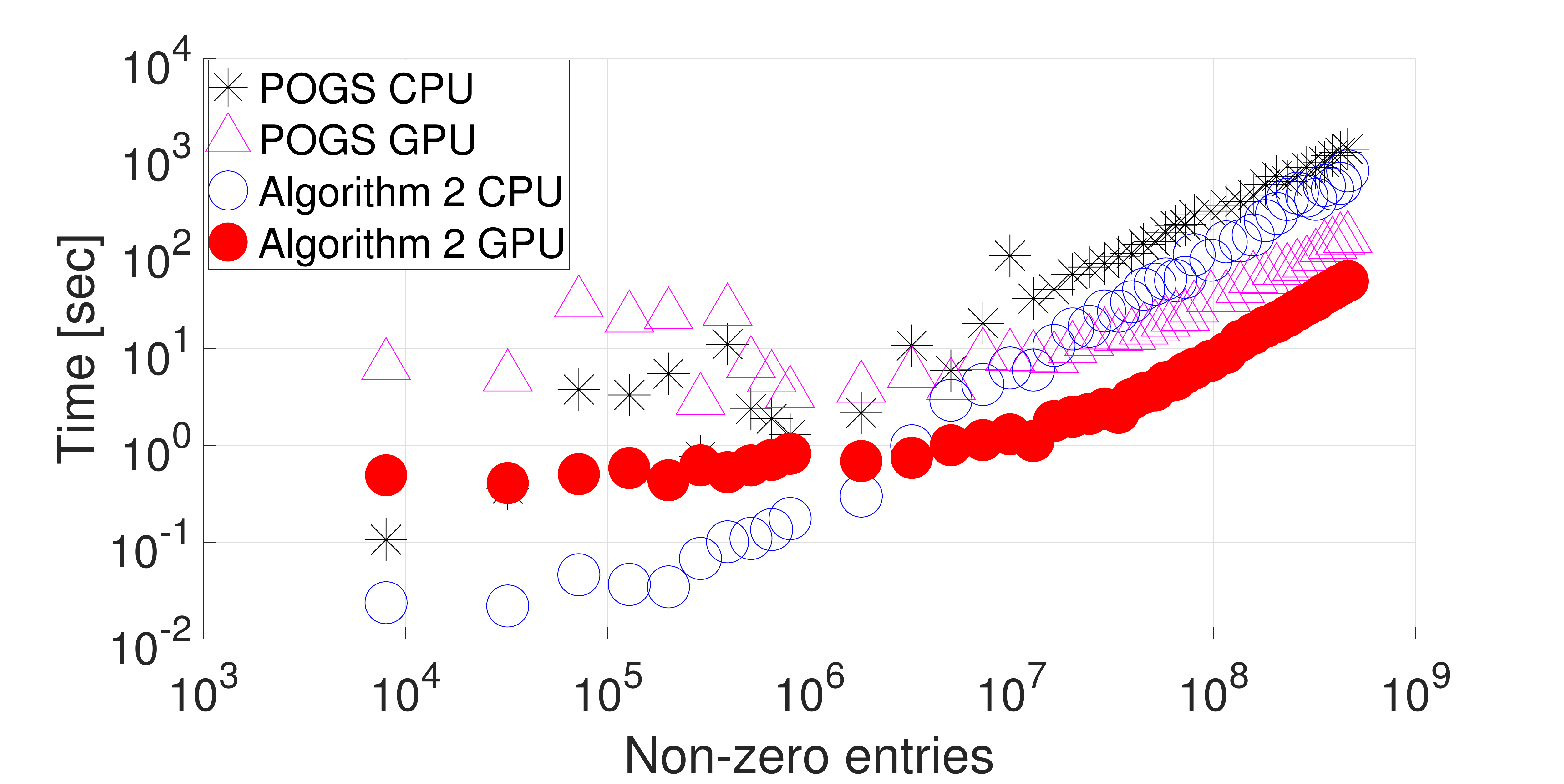}}
	\subfloat[\label{fig:lp_pogs2}] {\includegraphics[width =0.333\textwidth]{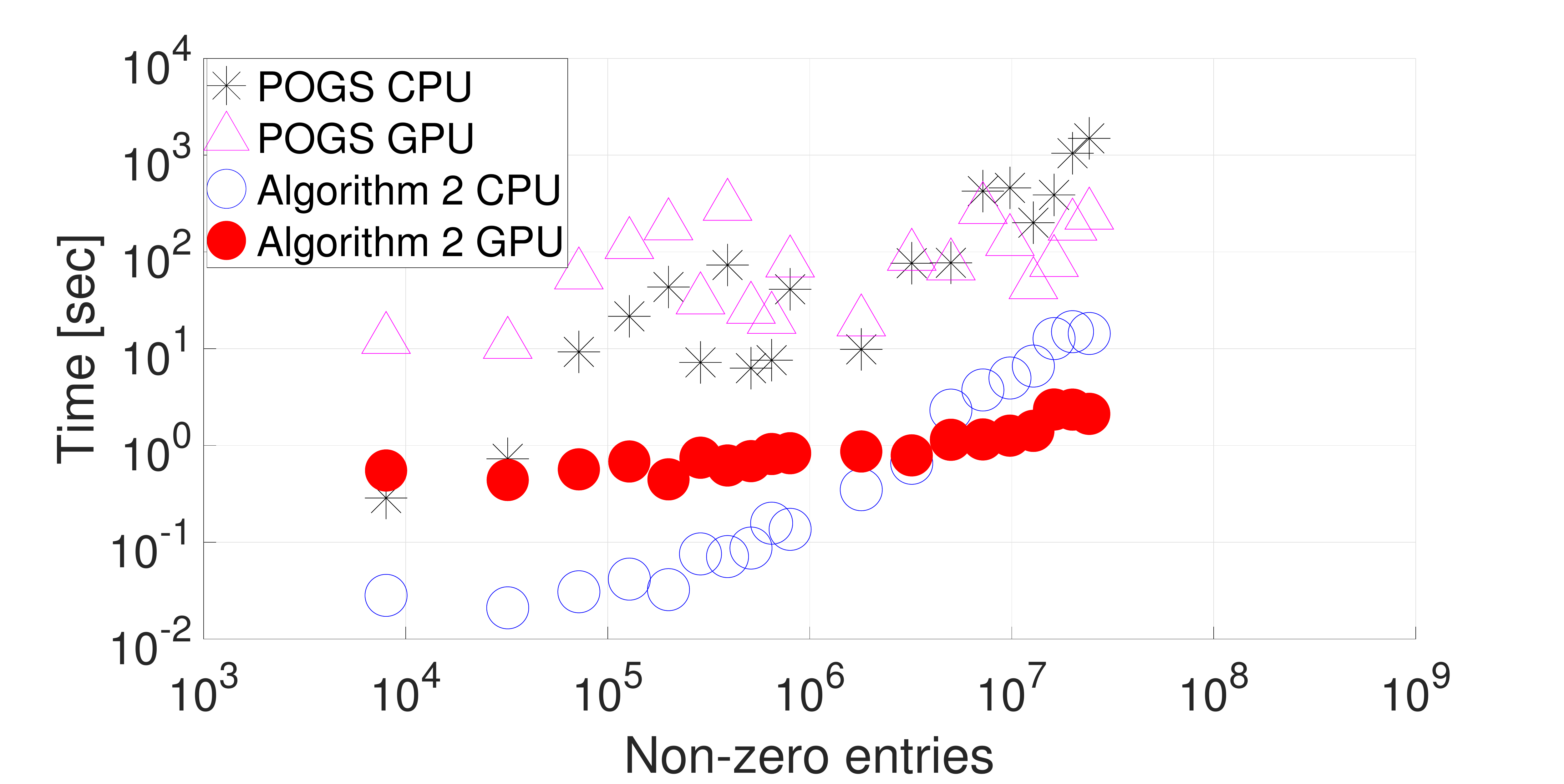}}
	\subfloat[\label{fig:lp_pogs3}] {\includegraphics[width =0.333\textwidth]{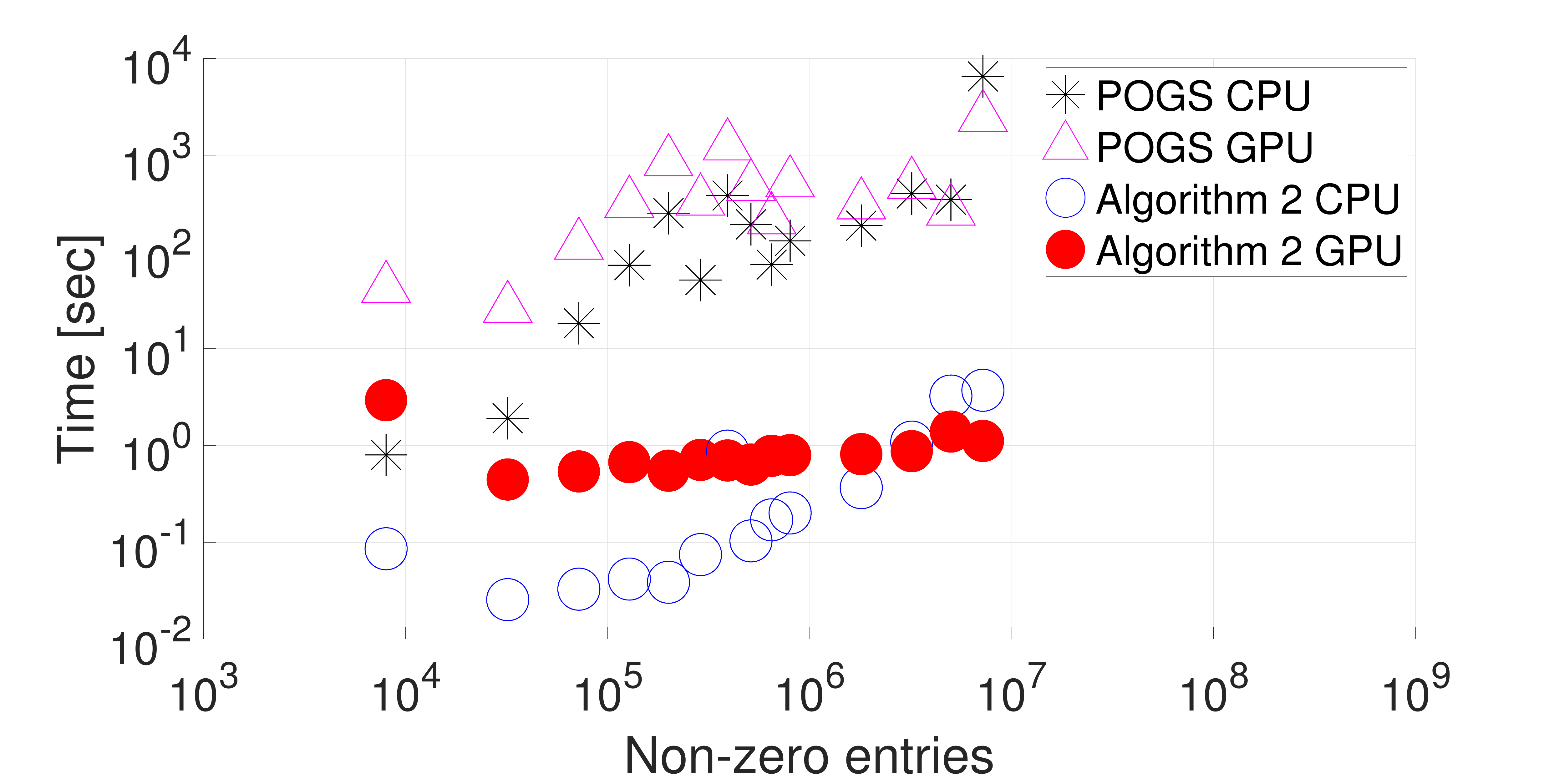}}
	\caption{The performance of Algorithm \ref{al:alg_O} for linear programming in comparison with POGS with the absolute and relative tolerances equal to (a) $\varepsilon^{\mathrm{abs}}=10^{-5}$, $\varepsilon^{\mathrm{rel}}=10^{-4}$, (b) $\varepsilon^{\mathrm{abs}}=10^{-6}$, $\varepsilon^{\mathrm{rel}}=10^{-5}$, and (c) $\varepsilon^{\mathrm{abs}}=10^{-7}$, $\varepsilon^{\mathrm{rel}}=10^{-6}$.}
	\label{fig:lp_pogs_time}
\end{figure*}

\subsection{Linear Programming}

\subsubsection{Comparisons with OSQP, Gurobi, and MOSEK}
This cases study is concerned with the class of linear programming problems. The performance of Algorithm \ref{al:alg_O} is {tested} in comparison with the solvers, OSQP, Gurobi, and MOSEK.  We have generated random LP instances with $n$ ranging from $100$ to $30000$, and $m = \lfloor0.8n\rfloor$. The number of nonzero elements of $\Abf$ ranges from $10^4$ to $10^9$. The data is generated as follows:
\begin{itemize}
	\item The elements of $\Abf\in\Rbb^{m\times n}$ have i.i.d uniform distribution from the interval $[-1,1]$.
	\item $\bbf: =\Abf|\dot{\xbf}|$ where the elements of $\dot{\xbf}\in\Rbb^{n}$ have i.i.d standard normal distribution.
	\item The elements of $\cbf\in\Rbb^n$ have i.i.d standard normal distribution.
	\item And $\Kcal=\Rbb_+^n$.
\end{itemize}
We start applying adaptive conditioning at iterations 300 and apply it {once again} in every 100 steps, i.e.,  
\begin{align}
\Lcal =\{300, 400, 500, \ldots\}.
\end{align}
Parameters $t$ and $\mu$ are set to 9.2 and 1, respectively.
The advantage of using GPU can be seen for large scale problems, when the problem size becomes larger, the GPU starts outperforming the other solvers significantly. The default settings are used for all three competing solvers and Algorithm \ref{al:alg_O} is terminated once a solution with better primal and dual feasibility is obtained, as defined in \eqref{stopcri}. As demonstrated in Figure \ref{fig:lp_computation_time}, for large instances, we have achieved {approximately 3.3 and 19 times improvements for CPU and GPU respectively,  in comparison with Gurobi}, and more than an order-of-magnitude time improvement in comparison with MOSEK and OSQP with their default settings. 

\begin{figure}[t] 
	\centering
	\includegraphics[width =0.333\textwidth]{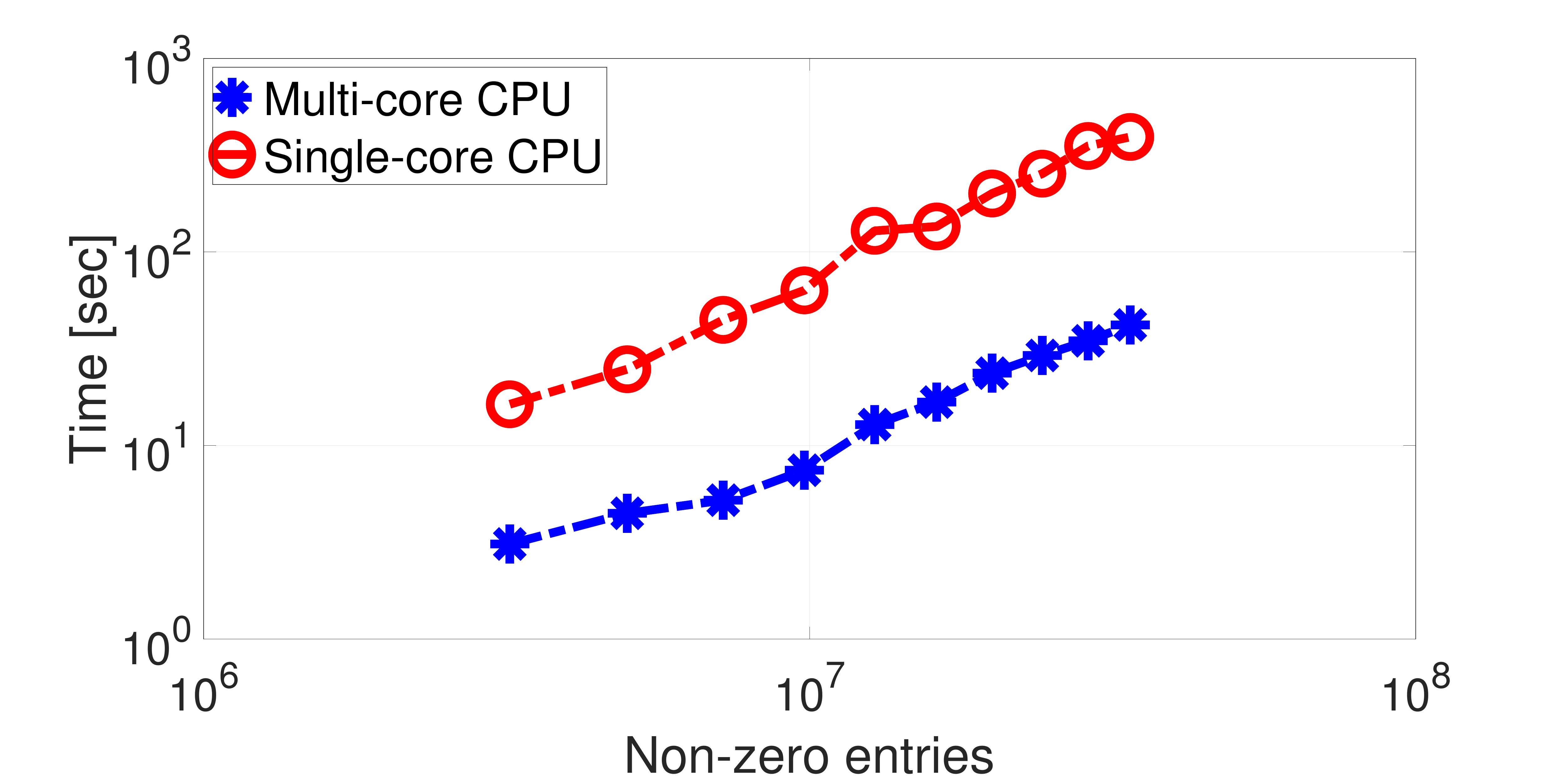}
	\caption{Single-core vs Multi-core CPU implementation}
	\label{fig:multi_core}
\end{figure}

{\subsection{Single vs Multi-core CPU Implementation}
	The iterative steps of Algorithms \ref{alg:1} and \ref{alg:O} are completely parallelizable. The parallel steps of algorithms are not only important for the graphics processing unit (GPU) implementation, but also provides the computational benefits for CPU implementation.  We demonstrate the parallel processing strength of the proposed algorithm  by solving  the previous instances of linear programming on both single-core and multi-core (20 cores) CPU settings. Figure \ref{fig:multi_core} shows that the multi-core CPU implementation is approximately 10 times faster than the single-core implementation. 
}

\subsection{Comparisons with POGS}
In this case study, we seek to demonstrate the ability of Algorithm \ref{al:alg_O} in finding very accurate solutions unlike competing first-order solvers that struggle with accuracy. In Figure \ref{fig:lp_pogs_time}, we perform comparisons between Algorithm \ref{al:alg_O} and the first-order solver POGS \cite{FB2018}. We use the default settings for POGS except for the absolute and relative tolerance values $\varepsilon^{\mathrm{abs}}$ and $\varepsilon^{\mathrm{rel}}$ for stopping criteria. Figure \ref{fig:lp_pogs_time} demonstrates that Algorithm \ref{al:alg_O} comprehensively outperforms one of the prominent first order solver, particularly with lower tolerance values. Similar to the previous experiment, the stopping criteria of Algorithm \eqref{al:alg_O} depends on the competing solver, as define in \eqref{stopcri}.

\begin{figure*}[t]
	\subfloat[\label{fig:socp_cone4_nonzero}]{ \includegraphics[width= 1.0\columnwidth, height= 0.25\textwidth]{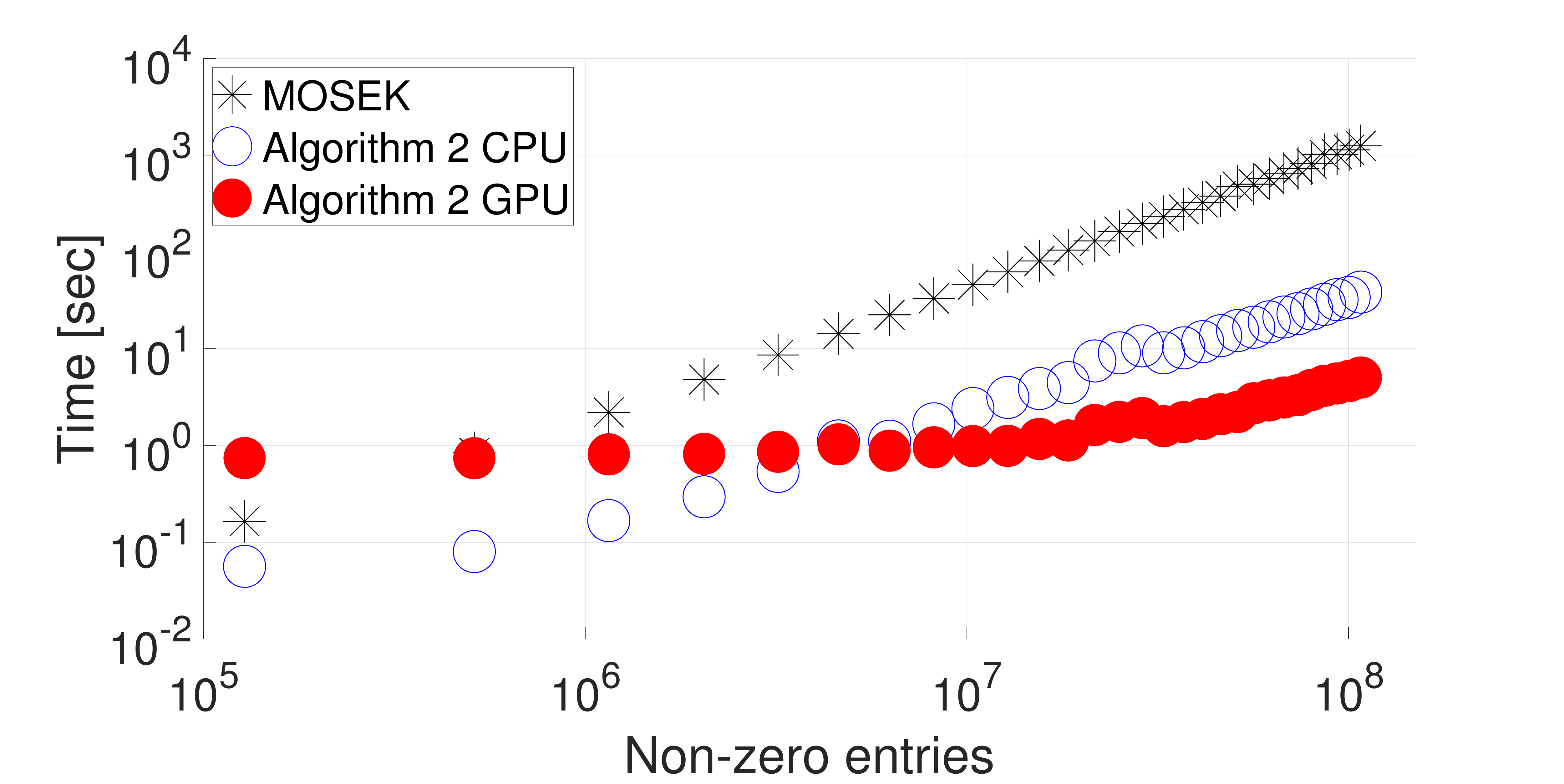}}
	\subfloat[\label{fig:socp_cone10_nonzero}]{\includegraphics[width= 1.0\columnwidth, height= 0.25\textwidth]{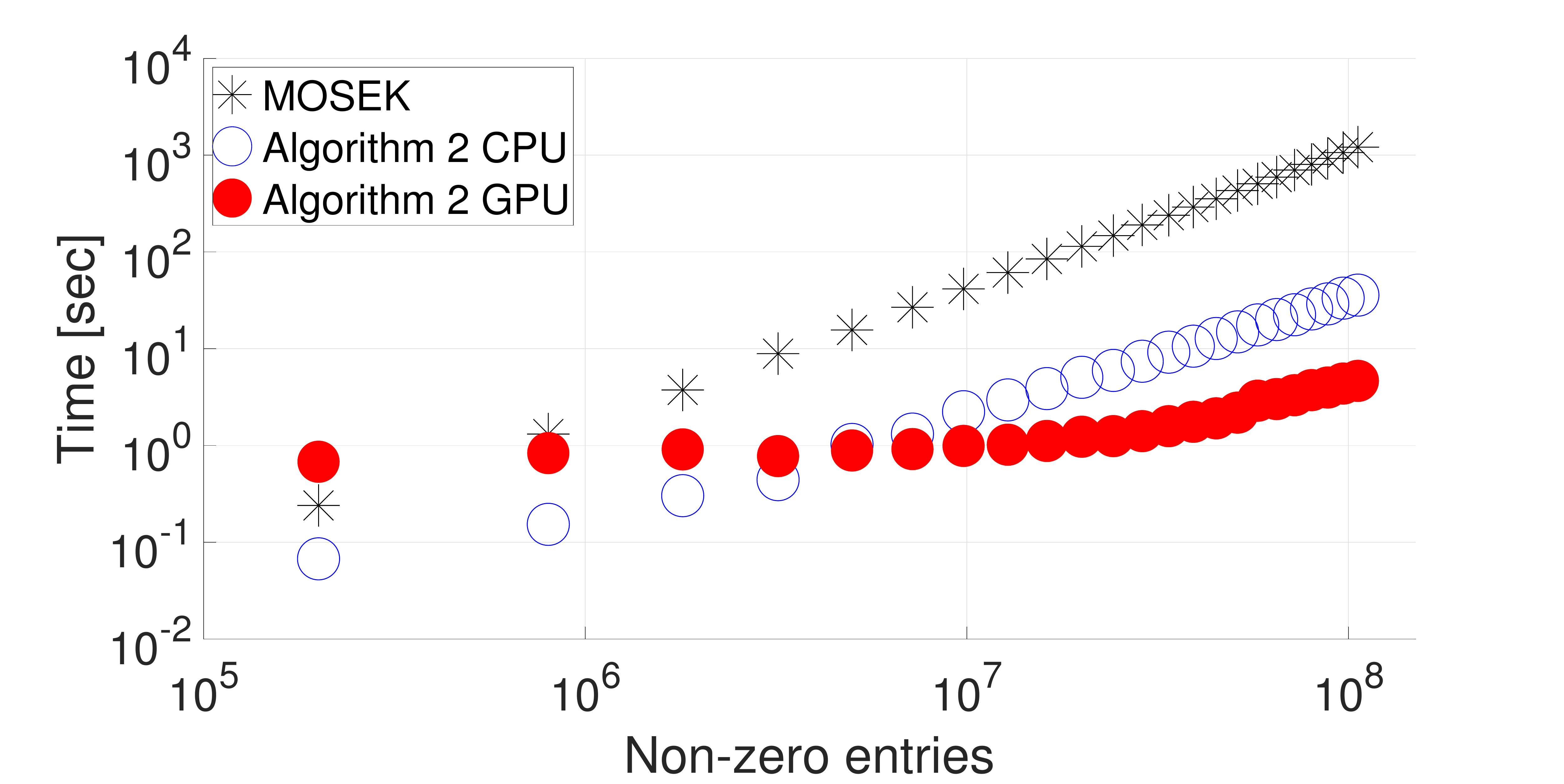}}
	\caption{The performance of Algorithm \ref{al:alg_O} for second order cone programming in comparison with MOSEK with Lorentz cones of size (a) $h=4$ and (b) $h =10$}
	\label{fig:socp_cone}
\end{figure*}

\subsection{Second-Order Cone Programming}

This {case} study is concerned with the class of second-order cone programming optimization problems. The performance of Algorithm \ref{al:alg_O} is tested in comparison with MOSEK on default settings.  We have generated random SOCP instances with $n$ ranging from $100$ to $2900$  {(MOSEK takes the maximum time of 1200 seconds)},  and $m = \lfloor0.8n\rfloor$: 
\begin{itemize}
	\item The elements of $\Abf\in\Rbb^{m\times n}$ have i.i.d uniform distribution from the interval $[-1,1]$.
	\item $\bbf: =\Abf\times\mathrm{abs}_{\Kcal}(\dot{\xbf})$ where the elements of $\dot{\xbf}\in\Rbb^{n}$ have i.i.d uniform distribution from the interval $[0,1]$.
	\item The elements of $\cbf\in\Rbb^n$ have i.i.d uniform distribution from the interval $[0,1]$.
	\item And $\Kcal=(\Kcal_h)^{\frac{n}{h}}$, where $\Kcal_h$ is the standard Lorentz cone of size $h$.
\end{itemize}
 We start applying adaptive conditioning at iterations {200} and apply it  {once again} in every 100 steps, i.e.,
\begin{align}
	\Lcal =\{200, 300, 400, \ldots\}.
\end{align}
Parameters $t$ and $\mu$ are set to 1.7 and 1, respectively.

The comparison of computational time for Lorentz cones of size $h=4$ and $h=10$ are reported in Figure \ref{fig:socp_cone}.  
It is clear from Figure \ref{fig:socp_cone} that Algorithm \eqref{al:alg_O} outperforms MOSEK by a large margins as the size of problem grows, 
while MOSEK performs better for smaller size problems. 

\section{Conclusions}\label{sec:concl}
We proposed a proximal numerical method with potential for parallelization. Next, an adaptive conditioning heuristic {was} developed to speed up the convergence of the proposed method. 
We provided a numerical example to demonstrate the fact that existing acceleration, parameter tuning and preconditioning methods have very limited effect on convergence behavior of first order methods. 
Moreover, we showed that convergence rate can be improved, irrespective of {the} condition number of data matrices.   
The proposed algorithm is implemented on {graphics processing unit} with an order-of-magnitude time improvement. 
A wide range of numerical experiments are conducted on large problems and results are compared with prominent first order solvers as well as the interior point method based solvers to {demonstrate} the claims made in this paper. 
We solved a variety of linear programs and second-order cone {programs. The} experimental results show that the proposed algorithm outperforms the first order algorithms in terms of computational time and achieves the accuracy levels comparable to second-order state-of-the-art methods.


	\bibliographystyle{IEEEtran}
	\bibliography{IEEEabrv,egbib_revised}

\appendix

In order to prove Theorem \ref{thm1}, we first give a few lemmas.

\begin{lemma}
Define the notation $|\cdot|_{\star}$ as
\begin{align}
	&|\cdot|_\star \triangleq\mathrm{abs}_{\Acal}(\mathrm{abs}_{\Kcal}(\cdot)).
\end{align}
Then for every $\ubf\in\Rbb^n$, we have
\begin{align}
	\||\ubf|_\star\|_2 = \|\ubf\|_2\label{norms}
\end{align}
and for every pair $\ubf,\vbf\in\Rbb^n$, we have
\begin{align}
	|\ubf|_\star^{\top} |\vbf|_\star\geq\ubf^{\top}\vbf.\label{normi}
\end{align}
\end{lemma}
\begin{proof}
The proof follows directly from the definition of  $\mathrm{abs}_{\Acal}$ and $\mathrm{abs}_{\Kcal}$.
\end{proof}

\begin{lemma}
Let $\xbf^{\mathrm{opt}}$ and $\zbf^{\mathrm{opt}}$ denote a pair of primal and dual solutions for problems \eqref{eq:prob_primal} and \eqref{eq:prob_dual}. Then
	\begin{align}
		&\sbf^{\mathrm{opt}} \triangleq \xbf^{\mathrm{opt}}-\mu\zbf^{\mathrm{opt}}
	\end{align}
is a fixed point of Algorithm \eqref{al:alg_1} and 
\begin{align}
\dbf=\dfrac{\sbf^{\mathrm{opt}}+|\sbf^{\mathrm{opt}}|_\star}{2}.\label{dddd}
\end{align}
\end{lemma}
\begin{proof}
According to the Karush–Kuhn–Tucker (KKT) optimality conditions, there exists $\ybf ^{\mathrm{opt}}\in\Rbb^m$, for which
\begin{subequations}
	\begin{align}
		&\cbf-\zbf ^{\mathrm{opt}} - \Abf^{\top}\ybf ^{\mathrm{opt}} =0\label{kkt1}\\
		&\Abf \xbf^{\mathrm{opt}}  = \bbf\label{kkt2}\\
		&(\xbf^{\mathrm{opt}})^{\top}\zbf^{\mathrm{opt}}=0,\qquad\xbf ^{\mathrm{opt}}\in \Kcal,\qquad 
		\zbf ^{\mathrm{opt}}\in \Kcal.\label{kkt3}
	\end{align}
\end{subequations}
Define
\begin{align}
&\pbf^{\mathrm{opt}} \triangleq \xbf^{\mathrm{opt}}+\mu\zbf^{\mathrm{opt}},
\end{align}
then according to \eqref{kkt3}, we have:
\begin{align}
\pbf^{\mathrm{opt}}=\mathrm{abs}_{\Kcal}(\sbf^{\mathrm{opt}}).\label{abs}
\end{align}
Hence
\begin{subequations}
\begin{align}
&\!\!\dbf-\dfrac{\sbf^{\mathrm{opt}}\!+\!
	|\sbf^{\mathrm{opt}}|_\star}{2}=\\
&\!\!\Abf^{\dagger}\bbf+\! \dfrac{\mu(\mathrm{abs}_{\Acal}(\cbf)\!-\!\cbf)}{2}\!-\!\dfrac{\sbf^{\mathrm{opt}}\!+\!
|\sbf^{\mathrm{opt}}|_\star}{2}\overset{\eqref{abs}}{=}\\
&\!\!\Abf^{\dagger}\bbf+ \dfrac{\mu(\mathrm{abs}_{\Acal}(\cbf)\!-\!\cbf)}{2}-\dfrac{\sbf^{\mathrm{opt}}+
	\mathrm{abs}_{\Acal}(\pbf^{\mathrm{opt}})}{2}=\label{20c}\\
&\!\!\Abf^{\dagger}(\bbf-\Abf\xbf^{\mathrm{opt}})-
\mu(\cbf-\zbf^{\mathrm{opt}}-\Abf^{\dagger}\Abf(\cbf-\zbf^{\mathrm{opt}}))\overset{\eqref{kkt2}}{=}\\
&\!\!
\mu(\cbf-\zbf^{\mathrm{opt}}-\Abf^{\dagger}\Abf(\cbf-\zbf^{\mathrm{opt}}))\overset{\eqref{kkt1}}{=}\\
&\!\!
\mu(\cbf-\zbf^{\mathrm{opt}}+\Abf^{\dagger}\Abf\Abf^{\top}\ybf^{\mathrm{opt}})=\label{20h}\\
&\!\!
\mu(\cbf-\zbf^{\mathrm{opt}}+\Abf^{\top}\ybf^{\mathrm{opt}})\overset{\eqref{kkt1}}{=}\zerobf_n,\label{20i}
\end{align}
which concludes \eqref{dddd}. Now according to the steps of Algorithm \ref{al:alg_1}, one can immediately conclude that $\sbf^{\mathrm{opt}}$ is a fixed point. 
\end{subequations}
\end{proof}

\begin{lemma}\label{lm:sequence}
Let $\{\sbf^l\}^{\infty}_{l=0}$ be the sequence generated by Algorithm \eqref{al:alg_1} and define $\sbf^{\mathrm{opt}} \triangleq \xbf^{\mathrm{opt}}-\mu\zbf^{\mathrm{opt}}$, where $\xbf^{\mathrm{opt}}$ and $\zbf^{\mathrm{opt}}$ denote an arbitrary pair of primal and dual solutions for problems \eqref{eq:prob_primal} and \eqref{eq:prob_dual}.
Then,
\begin{itemize}
\item[a)] the sequence $\{\|\sbf^{l}-\sbf^{\mathrm{opt}}\|_2\}^{\infty}_{l=0}$ is convergent,
\item[b)] and the sequence $\{\|\sbf^{l+1}-\sbf^{l}\|_2\}^{\infty}_{l=0}$ converges to zero.
\end{itemize}
\end{lemma}

\begin{proof}
According to the steps of Algorithm \ref{al:alg_1}, we have
\begin{align}
&\sbf^{l+1}=\dfrac{\sbf^l-|\sbf^l|_\star}{2} + \dbf\phantom{\Big|}
\end{align}
and due to \eqref{dddd}:
\begin{align}
&\sbf^{l+1}=\dfrac{\sbf^l+\sbf^{\mathrm{opt}}}{2} - \dfrac{|\sbf^l|_\star-|\sbf^{\mathrm{opt}}|_\star}{2}.\label{ssss}
\end{align}
Hence
\begin{subequations}\label{sequ}
\begin{align}
&\|\sbf^{l+1}-\sbf^{l}\|^2_2+\|\sbf^{l+1}-\sbf^{\mathrm{opt}}\|^2_2  - \|\sbf^{l}-\sbf^{\mathrm{opt}}\|^2_2\overset{\eqref{ssss}}{=}\\
&\|\dfrac{\sbf^l-\sbf^{\mathrm{opt}}}{2} + \dfrac{|\sbf^l|_\star-|\sbf^{\mathrm{opt}}|_\star}{2}\|^2_2 +\\ 
&\|\dfrac{\sbf^l-\sbf^{\mathrm{opt}}}{2} - \dfrac{|\sbf^l|_\star-|\sbf^{\mathrm{opt}}|_\star}{2}\|^2_2
- \|\sbf^{l}-\sbf^{\mathrm{opt}}\|^2_2=\\
&\dfrac{\||\sbf^l|_\star-|\sbf^{\mathrm{opt}}|_\star\|^2_2}{2}-\dfrac{\|\sbf^l-\sbf^{\mathrm{opt}}\|^2_2}{2}\overset{\eqref{norms}}{=}\\
&(\sbf^l)^{\top}\sbf^{\mathrm{opt}}-|\sbf^l|_\star^{\top}|\sbf^{\mathrm{opt}}|_\star\overset{\eqref{normi}}{\leq} 0,
\end{align}
\end{subequations}
which concludes that $\{\|\sbf^{l}-\sbf^{\mathrm{opt}}\|_2\}^{\infty}_{l=0}$ is nonincreasing and convergent. Additionally, \eqref{sequ} concludes that 
\begin{align}
	\sum_{l=0}^{\infty}\|\sbf^{l+1}-\sbf^{l}\|^2_2\;\leq\;\|\sbf^{0}-\sbf^{\mathrm{opt}}\|^2_2\label{series}
\end{align}
which means that $\{\|\sbf^{l+1}-\sbf^{l}\|_2\}^{\infty}_{l=0}$ converges to zero.
\end{proof}

\begin{proof}[Proof of theorem \ref{thm1}]
According to the first part of  Lemma \ref{lm:sequence}, the sequence $\{\sbf^l\}^{\infty}_{l=0}$ is bounded and therefore, it has a convergent subsequence $\{\bar{\sbf}_l\}^{\infty}_{l=0}$, where
\begin{align}
	\lim_{l\to\infty}{\bar{\sbf}_l}=\bar{\sbf}.\label{limlim}
\end{align}
Define
\begin{align}
\bar{\xbf}\triangleq \dfrac{\mathrm{abs}_{\Kcal}(\bar{\sbf})+\bar{\sbf}}{2}
\qquad\mathrm{and}\qquad
\bar{\zbf}\triangleq\dfrac{\mathrm{abs}_{\Kcal}(\bar{\sbf})-\bar{\sbf}}{2\mu}.\label{defxbar}
\end{align}
In order to show that $\bar{\xbf}$ and $\bar{\zbf}$ are a pair of primal and dual solutions, we prove the following KKT optimality criteria:
\begin{subequations}
	\begin{align}
		&\bar{\zbf } - \cbf \in\mathrm{range}\{\Abf^{\top}\}\label{kktt1}\\
		&\Abf \bar{\xbf}  = \bbf\label{kktt2}\\
		&\bar{\xbf}^{\top}\bar{\zbf}=0,\qquad\bar{\xbf}\in \Kcal,\qquad\bar{\zbf }\in \Kcal. \label{kktt3}
	\end{align}
\end{subequations}
Condition \eqref{kktt3} follows directly from the definition \eqref{defxbar}. Additionally, according to the second part of Lemma \ref{lm:sequence},
\begin{subequations}
\begin{align}
&\zerobf_n=\lim_{l\to\infty}{\sbf^{l+1}-\sbf^{l}}\\
&\phantom{\zerobf_n}=\lim_{l\to\infty}{ \dbf-\dfrac{\sbf^l+|\sbf^l|_\star}{2}}\\
&\phantom{\zerobf_n}=\lim_{l\to\infty}{ \dbf-\dfrac{\bar{\sbf}^l+|\bar{\sbf}^l|_\star}{2}}
\overset{\eqref{limlim}}{=}\dbf-\dfrac{\bar{\sbf}+|\bar{\sbf}|_\star}{2}
\end{align}
\end{subequations}
Hence,
\begin{align}
\dfrac{\bar{\sbf}+|\bar{\sbf}|_\star}{2}=\dbf
\overset{\eqref{dddd}}{=}\dfrac{\sbf^{\mathrm{opt}}+|\sbf^{\mathrm{opt}}|_\star}{2}
\end{align}
Therefore,
\begin{subequations}
\begin{align}
&\!\!\!\zerobf_n\!=\dfrac{\bar{\sbf}+|\bar{\sbf}|_\star}{2}-\dfrac{\sbf^{\mathrm{opt}}+|\sbf^{\mathrm{opt}}|_\star}{2}\\
&\!\!\!\phantom{\zerobf_n}\!=\dfrac{\bar{\sbf}-\sbf^{\mathrm{opt}}\!+\!(2\Abf^{\dagger}\Abf\!-\!\Ibf_n)(\mathrm{abs}_{\Kcal}(\bar{\sbf})-\mathrm{abs}_{\Kcal}(\sbf^{\mathrm{opt}}))}{2}\!\!\\
&\!\!\!\phantom{\zerobf_n}\!=\dfrac{(\bar{\xbf}-\xbf^{\mathrm{opt}})-\mu(\bar{\zbf}-\zbf^{\mathrm{opt}})}{2}\nonumber\\
&\!\!\!\phantom{\zerobf_n}\qquad+\dfrac{(2\Abf^{\dagger}\Abf-\Ibf_n)[(\bar{\xbf}-\xbf^{\mathrm{opt}})+\mu(\bar{\zbf}-\zbf^{\mathrm{opt}})]}{2}\\
&\!\!\!\phantom{\zerobf_n}\!=\Abf^{\dagger}(\Abf\bar{\xbf}-\Abf\xbf^{\mathrm{opt}})-\mu(\Ibf_n-\Abf^{\dagger}\Abf)(\bar{\zbf}-\zbf^{\mathrm{opt}}).
\end{align}
\end{subequations}
Now, pre-multiplication by $\Abf$ concludes that
\begin{align}
\Abf\bar{\xbf}=\Abf\xbf^{\mathrm{opt}}=\bbf.
\end{align}
Similarly,
\begin{subequations}
\begin{align}
&\!\!\!\zerobf_n\!=\dfrac{\bar{\sbf}+|\bar{\sbf}|_\star}{2}-\dbf\\
&\!\!\!\phantom{\zerobf_n}\!=\Abf^{\dagger}(\Abf\bar{\xbf}-\bbf)-
\mu(\Ibf_n-\Abf^{\dagger}\Abf)(\bar{\zbf}-\cbf)\\
&\!\!\!\phantom{\zerobf_n}\!=
\mu(\Ibf_n-\Abf^{\dagger}\Abf)(\cbf-\bar{\zbf})
\end{align}
\end{subequations}
which concludes \eqref{kktt1}. Therefore $\bar{\xbf}$ and $\bar{\zbf}$ are primal and dual optimal, and according to the first part of Lemma \ref{lm:sequence}, the following limit exists:
\begin{align}
\lim_{l\to\infty}\|\sbf^{l}-\bar{\sbf}\|_2=\lim_{l\to\infty}\|\bar{\sbf}^{l}-\bar{\sbf}\|_2=0
\end{align}
which completes the proof.
\end{proof}

\end{document}